\documentclass{article}
\usepackage{etex}
\usepackage[utf8]{inputenc}
\usepackage[english]{babel}
\usepackage[dvipsnames]{xcolor}
\usepackage{amsthm}
\usepackage{amsmath}
\usepackage{mathtools}
\usepackage{bbm}
\usepackage{algorithm} 
\usepackage[noend]{algpseudocode}
\usepackage{amssymb}
\usepackage{extarrows}
\usepackage{graphicx}
\usepackage{array}
\usepackage{booktabs}
\usepackage{tikz}
\usepackage[font=footnotesize,skip=0pt]{caption}
\usetikzlibrary{external}
\usepackage{pgfplots}
\usetikzlibrary{arrows.meta}
\usepgfplotslibrary{fillbetween}
\pgfplotsset{compat=newest}
\pgfplotsset{scaled x ticks=false} 
\pgfplotsset{
	small, 
	axis y line = left,
	axis x line = bottom,
	axis line style = {-latex, thick},
	xlabel style = {
		xshift = 14*\pgfkeysvalueof{/pgfplots/major tick length},
		yshift = 3*\pgfkeysvalueof{/pgfplots/major tick length},    
		anchor=north},  
	ylabel style = {
		yshift = -7*\pgfkeysvalueof{/pgfplots/major tick length},
		xshift = 15*\pgfkeysvalueof{/pgfplots/major tick length},
		rotate = -90,
		anchor=north},
} 

\pgfplotsset{
	axisStyle/.style={
		axis y line =left,
		axis x line =bottom,
		axis line style ={very thick},
	}
}

\newlength\figurewidth
\usepackage{float}
\usepackage{xurl}
\usepackage{color}
\usepackage[draft,nomargin]{fixme}
\fxsetup{theme=color,layout=inline} 
\usepackage{manfnt}                      
\usepackage{hyperref}
\usepackage{enumerate} 
\usepackage[absolute]{textpos}
\usepackage{ragged2e}
\usepackage{tabularx}
\usepackage{here}
\usepackage{caption}

\usepackage[normalem]{ulem} 

\allowdisplaybreaks

\newcommand{\RNum}[1]{\uppercase\expandafter{\romannumeral #1\relax}}

\newcommand{\NN}{\mathbb{N}}

\newcommand{\RR}{\mathbb{R}}

\newcommand{\EE}{\mathbb{E}}

\newcommand{\dom}{\mathrm{dom}\ }
\newcommand{\ri}{\mathrm{ri}\ }

\DeclareMathOperator*{\argmin}{\arg\!\min}

\newtheorem{thm}{Theorem}[section]
\newtheorem{defi}[thm]{Definition}
\newtheorem{prop}[thm]{Proposition}
\newtheorem{lem}[thm]{Lemma}
\newtheorem{bem}[thm]{Remark}
\newtheorem{as}{Assumption}

\newtheorem{bsp}[thm]{Example}
\newtheorem{cor}[thm]{Corollary}

\begin{document}
	
	\newcommand{\change}[2]{{\color{red}\sout{#1} {\color{ForestGreen} #2}}}
	
	\title{A Bregman-Kaczmarz method for nonlinear systems of equations}
	\author{
	Robert Gower
		\thanks{CCM, Flatiron Institute, Simons Foundation, \href{mailto:gowerrobert@gmail.com}{gowerrobert@gmail.com}} 
		\and 
	Dirk A. Lorenz
		\thanks{Institute of Analysis and Algebra, TU Braunschweig, \href{mailto:d.lorenz@tu-braunschweig.de}{d.lorenz@tu-braunschweig}}
		\thanks{Center for Industrial Mathematics, Fachbereich 3, University of Bremen, \href{mailto:d.lorenz@uni-bremen.de}{d.lorenz@uni-bremen.de}}
		\and
	Maximilian Winkler
		\thanks{Insitute of Analysis and Algebra, TU Braunschweig, \href{mailto:maximilian.winkler@tu-braunschweig.de}{maximilian.winkler@tu-braunschweig.de}}, 
		\thanks{Center for Industrial Mathematics, Fachbereich 3, University of Bremen, \href{mailto:maxwin@uni-bremen.de}{maxwin@uni-bremen.de}}
	} 
	\maketitle

	\begin{abstract}
		We propose a new randomized method for solving systems of nonlinear equations, which can find sparse solutions or solutions under certain simple constraints. The scheme only takes gradients of component functions and uses Bregman projections onto the solution space of a Newton equation.  
		In the special case of euclidean projections, the method is known as nonlinear Kaczmarz method. Furthermore if the component functions are nonnegative,
		we are in the setting of optimization under the interpolation assumption and the method reduces to SGD with the recently proposed stochastic Polyak step size. For general Bregman projections, our method is a stochastic mirror descent with a novel adaptive step size. We prove that in the convex setting
		each iteration of our method results in a smaller Bregman distance to exact solutions as compared to the standard Polyak step.
		Our generalization to Bregman projections comes with the price that a convex one-dimensional optimization problem needs to be solved in each iteration. This can typically be done with globalized Newton iterations.
		Convergence is proved in two classical settings of nonlinearity: for convex nonnegative functions and  locally for functions which fulfill the tangential cone condition.  
		Finally, we show examples in which the proposed method outperforms similar methods with the same memory requirements.
	\end{abstract}

	\noindent
	\textbf{AMS Classification:} 49M15, 90C53, 65Y20
	
	\noindent
	\textbf{Keywords:} Nonlinear systems, stochastic methods, randomized Kaczmarz, Bregman projections
	
	\section{Introduction}
	\label{sec:intro}
	
	
	
	We consider a constrained nonlinear system of equations
	\begin{align}
		\label{eqn:problem}
		f(x) = 0 \qquad \text{s.t. } x\in C,
	\end{align}
	where $f\colon D \subset\RR^d \to\RR^n $ is a nonlinear differentiable function and $C\subset D \subset \RR^d$ is a nonempty closed convex set. Let $S \subset C$ be the set of solutions of~\eqref{eqn:problem}. Our aim is to design an iterative method which approximates a solution of~\eqref{eqn:problem} and in each step uses first-order information of just a single component function $f_i$.

	The idea of our method is as follows. Given an appropriate convex function $\varphi\colon\RR^d\to \RR \cup \{+\infty\}$ with 
	\begin{align} \label{eq:domaincondition} 
		\overline{\dom\partial\varphi} = C, 
	\end{align}
	our method computes the \emph{Bregman projection} w.r.t. $\varphi$ onto the solution set of the local linearization of a component function $f_i$ around the current iterate $x_k$. 
	Here, the underlying distance is the \emph{Bregman distance} defined by
	\begin{align*}
		D_\varphi^{x^*}(x,y) = \varphi(y) - \varphi(x) - \langle x^*,y-x\rangle,
	\end{align*}
	where $x^*$ is a subgradient of $\varphi$ at $x$. 
	That is, the method we study is given by 
	\begin{align} \label{eq:ourmethodintro}
		x_{k+1} = &\arg\min_{x\in\RR^d} D_\varphi^{x_k^*}(x_k,x) \qquad \text{s.t. } x\in H_k,
	\end{align}
	where
	\begin{equation*}
		H_k := \{x\in\RR^d: f_{i_k}(x_k) + \langle \nabla f_{i_k}(x_k), x-x_k\rangle = 0\},
	\end{equation*} 
	where $i_k \in\{1,\ldots, n\}$ and $x_k^*$ is in the subgradient  $\partial \varphi(x_k).$
	Since one can show that Bregman projections are always contained in $\dom\partial\varphi$, the condition~\eqref{eq:domaincondition} guarantees that $x_k\in C$ holds for all $k$ and hence, if the $x_k$ converge, they converge to a point in $C$.
	In order for the Bregman projection $x_{k+1}$ to exist, we need that the hyperplanes $H_{k}$ have nonempty intersection with $\dom \varphi$. Proposition~\ref{prop:t_min_problem_abstract_hyperplane} below will show that the slightly stronger condition 
	\begin{align} 
		\label{eqn:Qualif_constraint_hyperplane}
		H_{k} \cap \ri \dom \varphi \neq\emptyset 
	\end{align}
	is sufficient for existence and uniqueness of the Bregman projection under regularity assumptions on $\varphi$. 
	If~\eqref{eqn:Qualif_constraint_hyperplane} is violated, we propose to compute a relaxed projection, which is always defined and inspired by the recently proposed \emph{mSPS method}~\cite{LLMO21}.

	\subsection{Related work and our contributions}
	\label{sec:related-work}
	
	\textbf{Nonlinear Kaczmarz method and Sparse Kaczmarz.}
	In the pioneering work by Stefan Kaczmarz~\cite{Kac37}, the idea of solving systems of equations by cycling through the separate equations and solving them incrementally was first executed on linear systems in finite dimensional spaces, an approach which is known henceforth as \emph{Kaczmarz method}. 
	In this conceptually simple method, an update is computed by selecting one equation of the system according to a rule that may be random, cyclic or adaptive, and computing an orthogonal projection onto its solution space, which is given by a hyperplane. 
	
	Recently, two completely different extensions of the Kaczmarz method have been developed. 
	One idea was to transfer the method to systems with nonlinear differentiable functions by considering its local linearizations: In each step $k$, an equation $i_k$ is chosen and the update $x_{k+1}$ is defined as the orthogonal projection
	\[ x_{k+1} = \argmin \|x-x_k\|_2^2 \quad \text{s.t.} \quad f_{i_k}(x_k)+\langle \nabla f_{i_k}(x_k),x-x_k\rangle = 0. \]
	It is easy to check that this update can be computed by
	\[ x_{k+1} = x_k - \frac{f_{i_k}(x_k)}{\|\nabla f_{i_k}(x_k)\|_2^2}\nabla f_{i_k}(x_k). \]
	This method was studied under the names \emph{Sketched Newton-Raphson}~\cite{GLY22} or \emph{Nonlinear Kaczmarz method}~\cite{WLBG22}. Convergence was shown for two kinds of mild nonlinearities, namely star convex functions~\cite{GLY22} and functions which obey a local tangential cone condition~\cite{WLBG22}.
	
	A different kind of extension of the Kaczmarz method has been proposed by~\cite{LSW14}. Here, the notion of projection was replaced by the (more general) Bregman projection, giving rise to the `sparse' Kaczmarz method, which can find sparse solutions of the system. The method has been further extended to inconsistent systems~\cite{LSTW22}, accelerated by block averaging~\cite{LT22} and investigated as a regularization method in Banach spaces~\cite{JLZ23}. But so far only linear systems have been addressed. 
	
	The present article unifies these two generalizations, that is, we study the case of nonlinearity and general Bregman projections onto linearizations and derive convergence rates in the two aforementioned nonlinear settings. We also demonstrate that instead of sparsity, the proposed method is able to handle simple constraints such as simplex constraints as well. 
	
	\textbf{Stochastic Polyak step size (SPS).} 
	One popular method for solving the finite-sum problem $\min \frac1n\sum_{i=1}^n \ell_i(x)$ is stochastic gradient descent (SGD), which is defined by the update $x_{k+1} = x_k - \gamma_k\nabla\ell_{i_k}(x_k)$. It is still a challenging question if there exist good choices of step sizes which are adaptive in the sense that no hyperparameter tuning is necessary. In this context, the \emph{stochastic Polyak step size} (SPS) 
	\begin{align}
		\label{eqn:SPS}
		\gamma_k = \frac{\ell_{i_k}(x_k)-\hat \ell_{i_k}}{c\cdot\|\nabla \ell_{i_k}(x_k)\|_2^2}
	\end{align}
	was proposed in~\cite{LVHL21}, where $c>0$ is a fixed constant and $\hat \ell_i = \inf \ell_i$. It was shown that the iterates of this method converge for convex lower bounded functions $f_i$ for which the \emph{interpolation} condition holds, meaning that there exists $\hat x\in\RR^d$ with 
	$
	\ell_i(\hat x) = \hat \ell_i
	$
	for all $i=1,...,n$. This assumption is strong, but can be fulfilled e.g. by modern machine learning applications such as non-parametric regression or over-parametrized deep neural networks~\cite{BBM18,BH21}. We cover these assumptions with our framework as a special case by requiring that the functions $f_i$ in~\eqref{eqn:problem} are nonnegative, which is clear by setting $f_i=\ell_i-\hat\ell_i$. The SPS method applied to $\ell_1,...,\ell_n$ then coincides with the Nonlinear Kaczmarz method applied to $f_1,...,f_n$.

	\textbf{Mirror Descent and SPS.} 
	For incorporating additional constraints or attraction to sparse solutions into SGD, a well-known alternative to projected SGD is the \emph{stochastic mirror descent} method (SMD)~\cite{AHL21,NY83,BBMZ17}, which is defined by the update
	\[ x_{k+1} \in \argmin_{x\in\RR^d} \ \gamma_k \langle \nabla f_{i_k}(x_k),x-x_k\rangle + D_\varphi^{x_k^*}(x_k,x). \]
	Here, $\varphi$ is a convex function with additional properties which will be refined later on, which is then called the \emph{distance generating function} (DGF), $x_k^*$ is a subgradient of $\varphi$ at $x_k$ and $D_\varphi$ is the Bregman distance associated to $\varphi$. 
	We demonstrate that our proposed method can be reinterpreted as mirror descent with a novel adaptive step size in case that the $f_i$ are nonnegative. Moreover, for $\varphi(x)=\frac12\|x\|_2^2$, we obtain back the SGD method with the stochastic Polyak step size. For general $\varphi$, computing the step size requires the solution of a convex one-dimensional minimization problem. This is a similar situation as in the update of the stochastic dual coordinate ascent method~\cite{shalev2013stochastic}, a popular stochastic variance reduced method for minimizing regularized general linear models.
	
	The two recent independent works~\cite{LLMO21} and~\cite{CLY22} propose to use the stochastic Polyak step size from SGD in mirror descent. This update has the advantage that it is relatively cheap to compute. However, we prove that for convex functions, our proposed method
	takes bigger steps in terms of Bregman distance towards the solution  of~\eqref{eqn:problem}.
	We generalize the step size from~\cite{LLMO21} to the case in which the functions $f_i$ are not necessarily nonnegative and employ this update as a \emph{relaxed projection} whenever our iteration is not defined. We compare our proposed method with the method which always performs relaxed projections in our convergence analysis and experiments. As an additional contribution, we improve the analysis for the method in~\cite{LLMO21} for the case of smooth strongly convex functions $f_i$ (Theorem~\ref{thm:Convergence_nonnegative_convex_Li_smooth_mu_sc}).
	
	Finally,
	our method is by definition scaling-invariant in the sense that a multiplicative change $\varphi \mapsto \alpha\varphi$ of the DGF $\varphi$ with a constant $\alpha>0$ does not affect the method. To the best of our knowledge, this is the first mirror descent method which has this property.
	
	\textbf{Bregman projection methods.}
	The idea of using Bregman projections algorithmically dates back to the seminal paper~\cite{Bregman67}, which proposed to solve the \emph{feasibility problem}
	\[ \text{find}\quad \hat x \in \bigcap_{i=1}^n C_i \]
	for convex sets $C_i$ by iterated Bregman projections onto the sets $C_i$. This idea initiated an active line of research~\cite{AB97,BBC03,BB97,BC03,BR06,CEH01,CL81,CR96,KS21,PRR19} with applications in fields such as matrix theory~\cite{DT08}, image processing~\cite{CE89,Ctin91,BGOWX05} and optimal transport~\cite{BCCP15, HLJ19}. We can view problem~\eqref{eqn:problem} as a feasibility problem by setting $C_i=\{x\mid f_i(x)=0\}$. Our approach to compute Bregman projections onto linearizations has already been proposed for the case of convex inequalities $C_i=\{x\mid f_i(x)\leq 0\}$ under the name \emph{outer Bregman projections}~\cite{BIZ03,BR01}. Convergence of this method was studied in general Banach spaces. Obviously, the two problems coincide for convex nonnegative functions $f_i$. However, to the best of our knowledge, convergence rates have been given recently only in the case that the $C_i$ are hyperplanes~\cite{KS21}. In this paper, we derive rates in the space $\RR^d$. Also, we extend our analysis to the nonconvex setting for equality constraints.
	
	\textbf{Bregman-Landweber methods.}
	There are a couple of works in inverse problems, typically studied in Banach spaces, which already incorporate Bregman projections into first-order methods with the aim of finding sparse solutions. Bregman projections were combined with the nonlinear Landweber iteration the first time in~\cite{LKS06}. Later,~\cite{BH12} employed Bregman projections for $L_1$- and TV-regularization. A different nonlinear Landweber iteration with Bregman projections for sparse solutions of inverse problems was investigated in~\cite{MS16}. All of these methods use the full Jacobian $Df(x)$ in each iteration. In~\cite{Jin16,JW13}, a deterministic Kaczmarz method incorporating convex penalties was proposed which performs a similar mirror update as our method, but with a different step size which does not originate from a Bregman projection. The apparently closest related method to our proposed one was recently suggested in~\cite{CGHT21}, where the step size was calculated as the solution of a quite similar optimization problem, which is still different and also does not come with the motivation of a Bregman projection. 
	
	\textbf{Sparse and Bregman-Newton methods.} 
	Finally, since our proposed method can be seen as a stochastic first-order Newton iteration, we briefly point out that a link of Newton's method to topics like Bregman distances and sparsity has already been established in the literature. Iusem and Solodov~\cite{IS97} introduced a regularization of Newton's method by a Bregman distance. Nesterov and Doikov~\cite{DN21} continued this work by introducing an additional nonsmooth convex regularizer. Polyak and Tremba~\cite{PT20} proposed a sparse Newton method which solves a minimum norm problem subject to the full Newton equation in each iteration, and presented an application in control theory~\cite{PT20_sparse}.

	\subsection{Notation}
	For a set $S\subset\RR^d$, we write its interior as $S^\circ$, its closure as $\overline{S}$ and its relative interior as $\mathrm{ri}(S)$. The Cartesian product of sets $S_i\subset\RR^d$, $i=1,...,m$, is written as $\bigtimes_{i=1}^m S_i$. The set $\mathrm{span}(S)$ is the linear space generated by all elements of $S$. We denote by $\mathbbm{1}_d$ the vector in $\RR^d$ with constant entries $1$. For two vectors $x,y\in\RR^d$, we express the componentwise (Hadamard) product as $x\cdot y$ and the componentwise logarithm and exponential as $\log(x)$ and $\exp(x)$. For a given norm $\|\cdot\|$ on $\RR^d$, by $\|\cdot\|_*$ we denote the corresponding \emph{dual norm}, which is given as
	\[ \|x\|_* = \sup_{\|y\|=1} \langle x,y\rangle, \qquad x\in\RR^d. \]

	\section{Basic notions and assumptions} 
	
	We collect some basic notions and results as well as our standing assumptions for problem~\eqref{alg:NBK}.
	
	\subsection{Convex analysis and standing assumptions}
	\label{sec:cvx-ana-basic-assumption}
	Let $\varphi\colon\RR^d\to\overline\RR :=\RR\cup\{+\infty\}$ be convex with 
	\[ \dom\varphi = \{x\in\RR^d: \varphi(x)<\infty\} \neq\emptyset. \]
	We also assume that $\varphi$ is \emph{lower semicontinuous}, i.e.
	$\varphi(x) \leq \liminf_{y\to x} \varphi(y)$
	holds for all $x\in\RR^d$, and \emph{supercoercive}, meaning that
	\[ \lim_{\|x\|\to\infty} \frac{\varphi(x)}{\|x\|} = +\infty. \]
	The \emph{subdifferential} at a point $x\in\dom\varphi$ is defined as 
	\[ \partial\varphi(x) = \big\{ x^*\in\RR^d: \varphi(x) + \langle x^*,y-x\rangle \leq \varphi(y) \quad \text{for all } y\in \dom\varphi\big\}. \]
	An element $x^*\in\partial\varphi(x)$ is called a \emph{subgradient} of $\varphi$ at $x$.
	The set of all points~$x$ with $\partial\varphi(x)\neq\emptyset$ is denoted by $\dom\partial\varphi$.  
	Note that the relative interior of $\dom\varphi$ is a convex set, while $\dom\partial\varphi$ may not be convex, for a counterexample see~\cite[p.218]{Rockafellar70}. 
	In general, convexity of $\varphi$ guarantees the inclusions 
	$\ri\dom\varphi \subset \dom\partial\varphi\subset \dom\varphi$.
	For later purposes, we require that
	$ \dom\partial\varphi = \ri\dom\varphi, $
	which will be fulfilled in all our examples.
	We further assume that $\varphi$ is \emph{essentially strictly convex}, i.e. strictly convex on $\ri\dom\varphi$. (In general, this property only means strict convexity on every convex subset of $\dom\partial\varphi$.). 
	The convex conjugate (or Fenchel-Moreau-conjugate) of $\varphi$ is defined by \[ \varphi^*(x^*) = \sup_{x\in\RR^d} \langle x^*,x\rangle - \varphi(x), \qquad x^*\in\RR^d. \] 
	The function $\varphi^*$ is convex and lower semicontinuous. Moreover, the essential strict convexity and supercoercivity imply that $\dom\varphi^*=\RR^d$ and $\varphi^*$ is differentiable, since $\varphi$ is essentially strictly convex and supercoercive, see~\cite[Proposition 14.15]{Bauschke11} and~\cite[Theorem 26.3]{Rockafellar70}.
	
	The \emph{Bregman distance} $D_\varphi^{x^*}(x,y)$ between $x,y\in\dom\varphi$ with respect to $\varphi$ and a subgradient $x^*\in\partial \varphi(x)$ is defined as
	\[ D_\varphi^{x^*}(x,y) = \varphi(y) - \varphi(x) - \langle x^*,y-x\rangle. \]
	Using Fenchel's equality $\varphi^*(x^*)=\langle x^*,x\rangle - \varphi(x)$ for $x^*\in\partial\varphi(x)$,
	one can rewrite the Bregman distance with the conjugate function as
	\begin{align} 
		\label{eqn:BregDist_with_conjugate}
		D_\varphi^{x^*}(x,y) = \varphi^*(x^*) - \langle x^*,y\rangle + \varphi(y). \end{align} 
	
	If $\varphi$ is differentiable at $x$, then the subdifferential $\partial\varphi(x)$ contains the single element $\nabla\varphi(x)$ and we can write
	\[ D_\varphi(x,y) := D^{\nabla\varphi(x)}_\varphi(x,y) = \varphi(y) - \varphi(x) - \langle \nabla\varphi(x),y-x\rangle. 
	\]

	The function $\varphi$ is called $\sigma$-\emph{strongly} convex w.r.t. a norm $\|\cdot\|$ for some $\sigma>0$, if for all $x,y\in\dom \partial\varphi$ it holds that $\frac{\sigma}{2}\|x-y\|^2\leq D_\varphi^{x^*}(x,y)$.

	In conclusion, we require the following standing assumptions for problem~\eqref{eqn:problem}:
	\begin{as}
		\label{as:Standard_assumptions}
		\mbox{}
		\begin{enumerate}[(i)]
			\item The set $C$ is nonempty, convex and closed.
			\item It holds that $\varphi\colon\RR^d\to\overline\RR$ is essentially strictly convex, lower semicontinuous and supercoercive. 
			\item The function $\varphi$ fulfills that $\overline{\dom\partial\varphi} = C$ and $\dom\partial\varphi = \ri\dom\varphi$. 
			\item For each $x\in\dom\varphi$ and each sequence $x_k\in\dom \partial\varphi$ with $x_k^*\in\partial\varphi(x_k)$ and $x_k\to x$ it holds that $D_\varphi^{x_k^*}(x_k,x)\to 0$.
			\item The function $f\colon D\to\RR^n$ is continuously differentiable with $D\supset C$.
			\item The set of solutions $S$ of~\eqref{eqn:problem} is non-empty, that is $S := C\cap f^{-1}(0)\neq \emptyset$.
		\end{enumerate}
	\end{as}

	\subsection{Bregman projections}
	\label{sec:bregman-projection}
	
	
	
	

	\begin{defi} Let $E\subset\RR^d$ be a nonempty convex set, $x\in\dom\partial\varphi$ and $x^*\in\partial\varphi(x)$. Assume that $E\cap\dom\varphi\neq\emptyset$. The \emph{Bregman projection} of $x$ onto~$E$ with respect to $\varphi$ and $x^*$ is the point $\Pi_{\varphi, E}^{x^*}(x)\in E\cap\dom\varphi$ such that 
		\begin{align*}
			D_\varphi^{x^*}\big(x, \Pi_{\varphi, E}^{x^*}(x)\big) = \min_{y\in E} D_\varphi^{x^*}(x,y). 
		\end{align*}
	\end{defi}
	
	Existence and uniqueness of the Bregman projection is guaranteed if $E\cap\dom\varphi\neq\emptyset$ by our standing assumptions due to the fact that the function $y\mapsto D_\varphi^{x^*}\big(x, y\big)$ is lower bounded by zero, coercive, lower semicontinuous and strictly convex. 
	For the standard quadratic $\varphi=\frac{1}{2}\|\cdot\|_2^2$, the Bregman projection is just the orthogonal projection.
	Note that if $E\cap\dom\varphi=\emptyset$, then for all $y\in E$ it holds that $D_\varphi^{x^*}(x,y)=+\infty$. 
	
	The Bregman projection can be characterized by variational inequalities, as the following lemma shows.
	
	\begin{lem}[\cite{LSW14}]
		\label{lem:BregmanProj_VarInequality}
		A point $z\in E$ is the Bregman projection of $x$ onto $E$ with respect to $\varphi$ and $x^*\in\partial\varphi(x)$ if and only if there exists $z^*\in\partial \varphi(z)$ such that one of the following conditions is fulfilled:
		\begin{enumerate}[(i)]
			\item $\langle z^*-x^*,z-y\rangle \leq 0 \qquad \text{for all } y\in E,$
			\item $D_\varphi^{z^*}(z,y) \leq D_\varphi^{x^*}(x,y) - D_\varphi^{x^*}(x,z) \qquad \text{for all } y\in E.$
		\end{enumerate}
	\end{lem}
	
	We consider Bregman projections onto hyperplanes 
	\[ H(\alpha,\beta):= \{x\in\RR^d: \langle \alpha,x\rangle = \beta\}, \qquad \alpha\in\RR^d, \ \beta\in\RR, \]
	and halfspaces
	\begin{align*}
		H^\leq(\alpha,\beta):= \{x\in\RR^d: \langle \alpha,x\rangle \leq \beta\}, \qquad \alpha\in\RR^d, \ \beta\in\RR,
	\end{align*}
	and analoguously we define $H^\geq(\alpha,\beta)$.
	
	The following proposition shows that the Bregman projection onto a hyperplane can be computed by solving a one-dimensional dual problem under a qualification constraint. We formulate this one-dimensional dual problem under slightly more general assumptions than previous versions, e.g. we neither assume smoothness of $\varphi$ (as e.g.~\cite{BB97,BC03, Bregman67, CL81, DT08}) nor strong convexity of $\varphi$ (as in~\cite{LSW14}).

	\begin{prop}
		\label{prop:t_min_problem_abstract_hyperplane}
		Let $\varphi$ fulfill Assumption~\ref{as:Standard_assumptions}(ii). Let $\alpha\in\RR^d\setminus\{0\}$ and $\beta\in\RR$ such that 
		\[ H(\alpha,\beta)\cap \ri\dom \varphi \neq\emptyset. \]
		Then, for all $x\in\dom\partial\varphi$ and $x^*\in\partial\varphi(x)$, the Bregman projection $\Pi_{\varphi, H(\alpha,\beta)}^{x^*}(x)$ exists and is unique. Moreover,  the Bregman projection is given by 
		\[ x_+:=\Pi_{\varphi, H(\alpha,\beta)}^{x^*}(x) = \nabla\varphi^*(x_+^*), \]
		where $x_+^* = x^* - \hat t\alpha \in \partial\varphi(x_+)$ and
		$\hat t$ is a solution to 
		\begin{align}
			\label{eqn:t_min_problem_abstract_hyperplane}
			\min_{t\in\RR} \varphi^*(x^*-t\alpha) + \beta t.
		\end{align}
	\end{prop}
	
	\begin{proof}
		The assumptions guarantee that $\varphi^*$ is finite and differentiable on the full space $\RR^d$. We already know that the Bregman projection $x_+$ exists and is unique. Fermat's condition applied to the projection problem $\min_{y\in H(\alpha,\beta)} D_\varphi^{x^*}(x,y)$ states that
		\[ 0 \in \partial \big(D_\varphi^{x^*}(x,\cdot) + \iota_{H(\alpha,\beta)} \big)(x_+), \]
		where the indicator function $\iota_M\colon\RR^d\to\overline\RR$ is defined by 
		\[ \iota_M(x) = \begin{cases}
			0, & x\in M, \\
			+\infty, & \text{otherwise.} 
		\end{cases} \]
		Applying subdifferential calculus~\cite[Theorem 23.8]{Rockafellar70}, where we make use of the fact that $H(\alpha,\beta)$ is a polyhedral set, we conclude that $x_+\in\dom\partial\varphi$ and
		\[ 0\in \partial\varphi(x_+) - x^* + \mathrm{span}(\{\alpha\}), \]
		where we used the fact that it holds $\partial\iota_{H(\alpha,\beta)} = \mathrm{span}(\{\alpha\})$ on $H(\alpha,\beta)$.
		Using subgradient inversion $(\partial\varphi)^{-1} = \nabla\varphi^*$, we arrive at the identity
		\[ x_+ = \nabla\varphi^*(x^*-\hat t\alpha) \]
		with some $\hat t\in\RR$. Inserting this equation into the constraint $\langle x_+,\alpha\rangle = \beta$, we conclude that $\hat t$ minimizes~\eqref{eqn:t_min_problem_abstract_hyperplane}.
	\end{proof}

	\section{Realizations of the method}
	\label{sec:Realizations}
	
	To solve problem~\eqref{eqn:problem}, we propose the following method. In each step, we randomly pick a component equation $f_{i_k}(x)=0$ and consider the set of zeros of its linearization around the current iterate $x_k$. This set is just the hyperplane
	\begin{align*}
		H_k := \big\{x\in\RR^d: f_{i_k}(x_k) + \langle \nabla f_{i_k}(x_k), x-x_k\rangle = 0\big\} = H(\nabla f_{i_k}(x_k), \beta_k),
	\end{align*}
	where
	\begin{align}
		\label{eqn:beta_k}
		\beta_k = \langle \nabla f_{i_k}(x_k),x_k\rangle - f_{i_k}(x_k).
	\end{align}
	For later purposes, we also consider the halfspace
	\[ H_k^\leq := \{ x\in\RR^d: f_{i_k}(x_k) + \langle f_{i_k}(x_k),x-x_k\rangle \leq 0 \}. \]
	As the update $x_{k+1}$, we now propose to take the Bregman projection of $x_k$ onto the set $H_k$ using Proposition~\ref{prop:t_min_problem_abstract_hyperplane}, which is possible if
	\begin{align}
		\label{eqn:Condition_hyperplane_nonempty_intersection}
		H_k \cap \dom\partial\varphi\neq\emptyset.
	\end{align}
	The update is then given by $x_{k+1}^* = x_k^*-t_{k,\varphi}\nabla f_{i_k}(x_k)$ and $x_{k+1}=\nabla\varphi^*(x_{k+1}^*)$ with
	\begin{align}
		\label{eqn:BregProj_stepsize}
		t_{k,\varphi} \in \argmin_{t\in\RR} \varphi^*(x_k^*-t\nabla f_{i_k}(x_k)) + \beta_k t.
	\end{align}
	Note that, although the Bregman projection $x_{k+1}$ is unique, $t_{k,\varphi}$ might not be unique. If~\eqref{eqn:Condition_hyperplane_nonempty_intersection} is not fulfilled, we define an update inspired from~\cite{LLMO21} by setting $x_{k+1} = \nabla\varphi^*(x_k^*-t_{k,\sigma}\nabla f_{i_k}(x_k))$
	with the Polyak-like step size
	\footnote{The typical setting in convergence analysis will be that $\varphi$ is $\sigma$-strongly convex with respect to a norm $\|\cdot\|$, and $\|\cdot\|_*$ will be its dual norm.}
	\begin{align}
		\label{eqn:mSPS_like_stepsize}
		t_{k,\sigma} = \sigma \frac{f_{i_k}(x_k)}{\|\nabla f_{i_k}(x_k)\|_*^2}
	\end{align}
	with some norm $\|\cdot\|_*$ and some constant $\sigma>0$.
	We will refer to the resulting update as the \emph{relaxed projection}. We note that it is always defined and gives a new point $x_{k+1}\in\dom\partial\varphi$.
	However, $x_{k+1}$ does not lie in $H_k$: Indeed, if it would lie in $H_k\cap\dom\partial\varphi$, this would contradict the assumption that~\eqref{eqn:Condition_hyperplane_nonempty_intersection} is not fulfilled. In~\cite{LLMO21}, the similar step size
	$t = \sigma \frac{f_{i_k}(x_k) - \inf_{i_k} f_{i_k}}{c\|\nabla f_{i_k}(x_k)\|_*^2}$
	with some constant $c>0$ was proposed for minimization with mirror descent under the name \emph{mirror-stochastic Polyak step size} (mSPS). 
	Both the projection and relaxed projection guarantee that $x_{k+1}\in\dom\varphi$ and deliver a new subgradient $x_{k+1}^*$ for the next update.
	The steps are summarized in Algorithm~\ref{alg:NBK} below. 
	
	\begin{algorithm}[htb]
		\begin{algorithmic}[1]
			\State Input: $\sigma>0$ and probabilities $p_i>0$ for $i=1,...,n$
			\State Initialization: $x_{0}^*\in \RR^d, x_0 = \nabla \varphi^*(x_0^*)$
			\For{$k=0,1,...$}
			\State choose $i_k\in\{1,...,n\}$ according to the probabilities $p_1,...,p_n$
			\If{$f_{i_k}(x_k)\neq 0$ and $\nabla f_{i_k}(x_k)\neq 0$} \\
			\Comment{otherwise, the component equation is solved already, or $H_k=\emptyset$}
			\State set $\beta_k = \langle \nabla f_{i_k}(x_k),x_k\rangle - f_{i_k}(x_k)$
			\If{ $H_k \cap \dom\partial\varphi\neq\emptyset$ }
			\State  \vspace{-0.3cm}
			$$ \displaystyle          \mbox{Find $t_k$:} \quad
			t_k \in \mathrm{argmin}_{t\in\RR} \, \varphi^*(x_k^*-t\nabla f_{i_k}(x_k)) + t\beta_k
			$$
			\vspace{-0.5cm}
			\Else{ set $t_k = \sigma \frac{f_{i_k}(x_k)}{\|\nabla f_{i_k}(x_k)\|_*^2}$ }
			\EndIf 
			\State update $x_{k+1}^*= x_k^* - t_k \nabla f_{i_k}(x_k) $
			\State update $x_{k+1} = \nabla \varphi^*(x_{k+1}^*)$
			\EndIf
			\EndFor
		\end{algorithmic}
		\caption{Nonlinear Bregman-Kaczmarz (NBK) method}
		\label{alg:NBK}
	\end{algorithm}	
	
	As an alternative method, we also consider the method which always chooses the step size $t_{k,\sigma}$ from~\eqref{eqn:mSPS_like_stepsize}.
	
	\begin{algorithm}[htb]
		\begin{algorithmic}[1]
			\State Input: $\sigma>0$ and probabilities $p_i>0$ for $i=1,...,n$ 
			\State Initialization: $x_{0}^*\in \RR^d, x_0 = \nabla \varphi^*(x_0^*)$
			\For{$k=0,1,...$}
			\State choose $i_k\in\{1,...,n\}$ according to the probabilities $p_1,...,p_n$              
			\If{$f_{i_k}(x)\neq 0$ and $\nabla f_{i_k}(x_k)\neq 0$}
			\State set $t_k = \sigma \frac{f_{i_k}(x_k)}{\|\nabla f_{i_k}(x_k)\|_*^2}$
			\State update $x_{k+1}^*= x_k^* - t_k \nabla f_{i_k}(x_k) $
			\State update $x_{k+1} = \nabla \varphi^*(x_{k+1}^*)$
			\EndIf
			\EndFor
		\end{algorithmic}
		\caption{Relaxed Nonlinear Bregman-Kaczmarz (rNBK) method}
		\label{alg:NBK_relaxed}
	\end{algorithm}

	Note that the problem~\eqref{eqn:BregProj_stepsize} is convex and one-dimensional and can be solved with the bisection method, if $\varphi^*$ is a $C^1$-function, or (globalized) Newton methods, if $\varphi^*$ is a $C^2$-function, see Appendix A. In Example~\ref{ex:varphi_from_sparse_Kaczmarz}, Example~\ref{ex:Simplex_entropy}, Example~\ref{ex:DGF_and_NBK_for_cartesian_products} and Example~\ref{ex:Double_simplex_entropy}, we show how to implement the steps of Algorithm~\ref{alg:NBK} for typical constraints. 
	
	\begin{bem}[Choice of $\sigma$ in Algorithm~\ref{alg:NBK} and Algorithm~\ref{alg:NBK_relaxed}]
			In this paper, we focus on the case that $\varphi$ is a strongly convex function. In this setting, we propose to choose the parameter $\sigma$ in Algorithm~\ref{alg:NBK} and Algorithm~\ref{alg:NBK_relaxed} as the modulus of strong convexity, since in this case our theorems in Section~\ref{sec:Realizations} guarantee convergence.
	\end{bem}

	\begin{bsp} \textup{(Unconstrained case and sparse Kaczmarz)}
		\label{ex:varphi_from_sparse_Kaczmarz} \mbox{} \\
		In the unconstrained case $\dom\varphi = \RR^d$, condition~\eqref{eqn:Condition_hyperplane_nonempty_intersection} is always fulfilled whenever $\nabla f_{i_k}(x_k)\neq 0$. \\
		For $\varphi(x) = \frac12\|x\|_2^2$, we obtain back the nonlinear Kaczmarz method~\textup{\cite{Nedic11,WLBG22,GLY22}}
		\[ x_{k+1} = x_k - \frac{f_{i_k}(x_k)}{\|\nabla f_{i_k}(x_k)\|_2^2}\nabla f_{i_k}(x_k). \]
		For the function 
		\begin{align}
			\label{eqn:RSK_mirror_map}
			\varphi(x)=\lambda\|x\|_1 + \frac12 \|x\|_2^2,
		\end{align} 
		Assumptions~\ref{as:Standard_assumptions}(i-iv) are fulfilled and it holds that
		$\varphi^*(x) = \tfrac12 \|S_\lambda(x)\|_2^2$ with the soft-shrinkage function 
		\[ S_\lambda(x) = \begin{cases} 
			x+\lambda, & x<-\lambda, \\
			0, & |x|\leq \lambda, \\
			x-\lambda, & x>\lambda
		\end{cases} 
		\]
		Hence, in this case lines 9, 11 and 12 of Algorithm~\ref{alg:NBK} read
		\begin{align*} 
			&\text{find } t_k\in \argmin_{t\in\RR} \beta_k t + \frac12 \|S_\lambda(x_k^*-t\nabla f_{i_k}(x_k))\|_2^2, \\
			&\text{update } x_{k+1}^* = x_k^* - t_k\nabla f_{i_k}(x_k), \\
			&\text{update } x_{k+1} = S_\lambda(x_{k+1}^*).
		\end{align*}
		
		For affine functions $f_{i}(x) = \langle a^{(i)},x\rangle - b_{i}$ with $a^{(i)}\in\RR^d$, $b_{i}\in\RR$, Algorithm~\ref{alg:NBK} has been studied under the name \emph{Sparse Kaczmarz method} and converges to a sparse solution of the linear system $f(x)=0$, see~\textup{\cite{LSW14,LS19}}. The update with $t_k$ from~\eqref{eqn:BregProj_stepsize} is also called the \emph{Exact step Sparse Kaczmarz} method. The linesearch problem can be solved exactly with reasonable effort, as $\varphi^*$ is a continuous piecewise quadratic function with at most $2d$ discontinuities. The corresponding solver is based on a sorting procedure and has complexity $\mathcal O(d\log(d))$, see~\textup{\cite{LSW14}} for details. 
	\end{bsp}

	\begin{bsp}\textup{(Simplex constraints)}
		\label{ex:Simplex_entropy} 
		We consider the probability simplex
		\[ C = \Delta^{d-1} := \big\{x\in\RR_{\geq 0}^d:\ \sum_{i=1}^d x_i=1\big\}. \]
		The restriction of the negative entropy function
		\begin{align}
			\label{eqn:NegativeEntropyFunctionSimplex}
			{\varphi}(x) = \begin{cases}
				\sum_{i=1}^d x_i \log(x_i), & x\in\Delta^{d-1} , \\
				+\infty, & \text{otherwise.}
			\end{cases}
		\end{align}
		fulfills Assumption~\ref{as:Standard_assumptions}(i-iv) and is $1$-strongly convex with respect to $\|\cdot\|_1$ due to Pinsker's inequality, see~\cite{FH03,Pinsker64} and \cite[Example 5.27]{Beck17}. We have that 
		\[ \dom\partial\varphi = \ri\Delta^{d-1} = \big\{x\in\RR_{> 0}^d:\ \sum_{i=1}^d x_i=1\big\} =: \Delta^{d-1}_+. \]
		We can characterize condition~\eqref{eqn:Condition_hyperplane_nonempty_intersection} easily as follows: The hyperplane $H(\alpha,\beta)$ intersects $\dom \partial\varphi=\Delta^{d-1}_+$ if and only if
		\begin{itemize}
			\item $\alpha = \beta \mathbbm{1}_d$ or
			\item there exist $r,s\in\{1,..., d\}$ with $\alpha_r<\beta<\alpha_s$.
		\end{itemize}
		This condition is quickly established by the intermediate value theorem and can be easily checked during the method. When verifying the condition in practice, in case of instabilities one may consider the restricted index set 
		\begin{align*}
			\{i=1,...,n\mid |x_i|>\delta\}
		\end{align*} for some positive $\delta$.
		
		The Bregman distance induced by $\varphi$ is the Kullback-Leibler divergence for probability vectors
		\[ D_{\varphi}(x,y) = \sum_{i=1}^d y_i\log\big(\frac{y_i}{x_i}\big), \qquad x\in\Delta^{d-1}_+, y\in\Delta^{d-1}. \]
		The convex conjugate of $\varphi$ is the \emph{log-sum-exp}-function
		$ {\varphi}^*(p) = \log\big(\sum\limits_{i=1}^d e^{p_i}\big). $
		Since $\varphi$ is differentiable, the steps of Algorithm~\ref{alg:NBK} can be rewritten by substituting $x_k^*$ by $\nabla\varphi(x_k)=1+\log(x_k)$. Denoting the $i$th component of an iterate $x_{l}$ by $x_{l,i}$, lines 9, 11 and 12 of Algorithm~\ref{alg:NBK} read
		\begin{align} 
			\label{eqn:t_KL_Simplex_min_problem}
			\text{find } t_k&\in\argmin_{t\in\RR} \beta_k t + \log\big(\sum_{i=1}^d x_{k,i} e^{-t\partial_i f_{i_k}(x_k)}\big), \\
			\label{eq:expgrad}
			x_{k+1} &=  \frac{x_k \cdot e^{-t_k\nabla f_{i_k}(x_k)}}{\|x_k \cdot e^{-t_k\nabla f_{i_k}(x_k)}\|_1},
		\end{align}
		where multiplication and exponentiation of vectors are understood componentwise.
		The method~\eqref{eq:expgrad} is known with the name exponentiated gradient descent or entropic mirror descent, provided that $t_k$ is nonnegative.
		We claim that our proposed step size $t_{k,\varphi}$ is new. Note that some convex polyhedra, such as $\ell^1$-balls, can be transformed to $\Delta^{d'-1}$ for some $d'\in\NN$ by writing a point as a convex combination of certain extreme points~\textup{\cite{LLMO21}}.
	\end{bsp}

	\begin{bsp}\textup{(Cartesian products of constraints)}
		\label{ex:DGF_and_NBK_for_cartesian_products}
		For $i\in\{1,...,m\}$, let $\varphi_i$ be a DGF for $C_i\subset D_i\subset\RR^{d_i}$ fulfilling Assumption~\ref{as:Standard_assumptions}(i-iv) and let $f\colon D:=\bigtimes_{i=1}^m D_i \to\RR^n$ fulfill Assumption~\ref{as:Standard_assumptions}(v-vi). Then 
		\begin{align*}
			\varphi(x) = \sum_{i=1}^m \varphi_i(x_i), \quad x = (x_1, ..., x_m) \text{ with } x_i\in \RR^d
		\end{align*}
		is a DGF for $C = \bigtimes_{i=1}^m C_i$ fulfilling Assumption~\ref{as:Standard_assumptions}(i-iv) with 
		\begin{align*}
			\dom\partial\varphi = \bigtimes_{i=1}^m \dom\partial\varphi_i \quad \text{and} \quad \partial\varphi(x) =  \bigtimes_{i=1}^m \partial\varphi_i(x_i) \text{ for all } x_i\in\dom\partial\varphi_i.
		\end{align*} 
		Denoting the $i$th component of an iterate $x^{(*)}_l$ by $x^{(*)}_{l,i}$, the lines 9, 11 and 12 of Algorithm~\ref{alg:NBK} for $i\in\{1,...,m\}$ read as follows:
		\begin{align*}
			\text{find } t_k&\in \argmin_{t\in\RR} \beta_k t + \sum_{i=1}^m \varphi_i^*\big(x_{k,i}^*-t\nabla_{i}f_{i_k}(x_k)\big), \\
			x_{k+1,i}^* &= x_{k,i}^* - t_k\nabla_i f_{i_k}(x_k) \qquad \text{for } i = 1,...,m, \\
			x_{k+1,i} &= \nabla\varphi_i^*(x_{k+1,i}^*) \qquad \text{for } i = 1,...,m,
		\end{align*}
		where $\nabla_i$ stands for the gradient w.r.t. the $i$th block of variables. \\
		Finally, we give a suggestion which constant $\sigma$ and norm $\|\cdot\|_\infty$ should be used in Algorithm~\ref{alg:NBK_relaxed}/ line 10 in Algorithm~\ref{alg:NBK}. Let us assume that $\varphi$ is $\sigma_i$-strongly convex w.r.t. a norm $\|\cdot\|_{(i)}$ on $\RR^{d_i}$. Then, the function $\varphi$ is $\sigma$-strongly convex with $\sigma=\min_{i=1,...,m}\sigma_i$ w.r.t. the mixed norm 
		\begin{align*}
			\|u\| := \sqrt{\sum_{i=1}^m \|u_i\|_{(i)}^2}.
		\end{align*}
		Indeed, for all $x,y$ with $x_i,y_i\in\RR^{d_i}$ it holds that 
		\begin{align*}
			\frac{\sigma}{2}\|x-y\|^2 = \frac{\sigma}{2}\sum_{i=1}^m \|x_i-y_i\|_{(i)}^2 \leq \sum_{i=1}^m D_{\varphi_i}^{x_i^*}(x_i,y_i) = D_\varphi^{x^*}(x,y).
		\end{align*}
		A quick calculation using Cauchy-Schwarz' inequality shows that the dual norm of $\|\cdot\|$ is given by 
		\begin{align}
			\label{eqn:Prod_dual_norm}
			\|u^*\|_* := \sqrt{\sum_{i=1}^m \|u_i^*\|_{(i,*)}^2},
		\end{align}
		where $\|\cdot\|_{(i,*)}$ is the dual norm of $\|\cdot\|_i$ on $\RR^{d_i}$. 
		Hence, we recommend to use Algorithm~\ref{alg:NBK_relaxed} with~\eqref{eqn:Prod_dual_norm} and $\sigma=\min\limits_{i=1,...,m}\sigma_i$, if $\varphi_i$ is $\sigma_i$-strongly convex w.r.t.~$\|\cdot\|_{(i)}$.
	\end{bsp}
	
	\begin{bsp}\textup{(Two-fold Cartesian product of simplex constraints)}
		\label{ex:Double_simplex_entropy}
		As a particular instance of Example~\ref{ex:DGF_and_NBK_for_cartesian_products} we consider the $2$-fold product of the probability simplex $C_i=\Delta^{d-1}$, $i\in\{1,2\}$ with $\varphi_i=\varphi$ from Example~\ref{ex:Simplex_entropy}. The properties from Assumption~\ref{as:Standard_assumptions} are inherited from the $\varphi_i$. We denote the iterates of Algorithm~\ref{alg:NBK} by $x_k,y_k\in\Delta^{d-1}$ and address its components by $x_{k,i}, y_{k,i}$ for $i=1,...,d$. Similar to Example~\ref{ex:Simplex_entropy}, the steps of the method can be rewritten as 
		\begin{align}
			\text{find } t_k &\in \argmin_{t\in\RR} \beta_k t + \log\Big(\sum_{l=1}^d x_{k,l} e^{-t(\nabla_x f_{i_k}(x_k))_l}\Big) + \log\Big(\sum_{l=1}^d y_{k,l} e^{-t(\nabla_y f_{i_k}(x_k))_l}\Big), \label{eqn:Prod_simplex_linesearch_obj} \\
			x_{k+1} &= \frac{x_k \cdot e^{-t_k\nabla_x f_{i_k}(x_k)}}{\|x_k \cdot e^{-t_k\nabla_x f_{i_k}(x_k)}\|_1}, \qquad 
			y_{k+1} = \frac{y_k \cdot  e^{-t_k\nabla_y f_{i_k}(y_k)}}{\|y_k \cdot e^{-t_k\nabla_y f_{i_k}(y_k)}\|_1},
			\nonumber
		\end{align}
		where $\nabla_x$ stands for the gradient w.r.t. $x$ and $\nabla_y$ for the gradient w.r.t. $y$.
		Also here, we can give a characterization of condition~\eqref{eqn:Condition_hyperplane_nonempty_intersection}: For $\alpha=(\alpha_1, \alpha_2)$ with $\alpha_1, \alpha_2\in\RR^d$ and $\beta\in\RR$, the hyperplane $H(\alpha,\beta)$ intersects $\dom\partial\varphi = \Delta_+^{d+1}\times\Delta_+^{d+1}$ if and only if for $(i,j)=(1,2)$ or $(i,j)=(2,1)$ one of the following conditions is fulfilled:
		\begin{itemize}
			\item $\alpha_i = c\mathbbm{1}_d$ with some $c\in\RR$ and $\alpha_j= (\beta-c)\mathbbm{1}_d$, 
			\item $\alpha_i = c\mathbbm{1}_d$ with some $c\in\RR$ and there exist $r, s\in\{1,...,d\}$ with \\ $\alpha_{j,r} < \beta - c < \alpha_{j,s}$ or
			\item $\big] \min \alpha_i, \max \alpha_i \big[ \ \cap \ 
			\big] \beta - \max \alpha_j, \beta - \min \alpha_j \big[ \ \neq \ \emptyset$.
		\end{itemize} 
		To prove this condition, we can invoke Proposition~\ref{prop:t_min_problem_abstract_hyperplane} which states that \eqref{eqn:Condition_hyperplane_nonempty_intersection} is fulfilled if and only if the objective function $g$ from~\eqref{eqn:t_min_problem_abstract_hyperplane} has a minimizer. Next, we note that, for each $c\in\RR$, the objective in~\eqref{eqn:Prod_simplex_linesearch_obj} can be rewritten as
		\begin{align*}
			g(t) &= \log\Big(\sum_{l=1}^d x_{k,l} e^{-(\alpha_{1,l}-c)t}\Big) + 
			\log\Big(\sum_{l=1}^d y_{k,l} e^{(\beta-c-\alpha_{2,l})t}\Big) \\
			&= \log\Big(\sum_{l=1}^d y_{k,l} e^{-(\alpha_{2,l}-c)t}\Big) + 
			\log\Big(\sum_{l=1}^d x_{k,l} e^{(\beta-c-\alpha_{1,l})t}\Big)
		\end{align*}
		and a case-by-case analysis shows that $g$ has a minimizer if and only if one of the above assertions is fulfilled. Note that also the here discussed condition can be easily checked during the method. We remind that, in case of instabilities one may consider the restricted index set  \[ \{i=1,...,n\mid |x_i|>\delta \ \text{and} \ |y_i|>\delta\} \] for some positive $\delta$. 
		As derived in Example~\ref{ex:DGF_and_NBK_for_cartesian_products}, in Algorithm~\ref{alg:NBK_relaxed}/ line 10 of Algorithm~\ref{alg:NBK} we use $\sigma=1$ and $\|u^*\|_* = \sqrt{\|u_1^*\|_\infty^2 + \|u_2^*\|_\infty^2}$.
	\end{bsp}

	\section{Convergence}
	
	In this section we do the convergence analysis of Algorithm~\ref{alg:NBK}. 
	
	At first, we characterize fixed points of Algorithm~\ref{alg:NBK} and Algorithm~\ref{alg:NBK_relaxed} and provide necessary lemmas for the subsequent analysis. In Section~\ref{subsec:star-convex}, we prove that for nonnegative star-convex functions $f_i$, condition~\eqref{eqn:Condition_hyperplane_nonempty_intersection} is always fulfilled and the step size~\eqref{eqn:BregProj_stepsize} is better than the relaxed step size~\eqref{eqn:mSPS_like_stepsize} in the sense that it results in an iterate with a smaller Bregman distance to all solutions of~\eqref{eqn:problem}. Finally, we present convergence results for Algorithm~\ref{alg:NBK} for this setting. 
	In Section~\ref{subsec:tcc}, we prove convergence in a second setting, namely in the case that the functions $f_i$ fulfill a local tangential cone condition as in~\cite{JW13,MS16,WLBG22}. \\

	As a first result, we determine the fixed points of Algorithm~\ref{alg:NBK}. The proposition states in particular that, in the unconstrained case $\dom\varphi=\RR^d$, fixed points are exactly the stationary points of the least-squares function $\|f(x)\|_2^2$.

	\begin{prop}
		\label{prop:Fixed_points}
		Let Assumption~\ref{as:Standard_assumptions} hold and let $x\in\dom\partial\varphi$ and $x^*\in\partial\varphi(x)$. The pair $\big(x,x^*\big)$ is a fixed point of Algorithm~\ref{alg:NBK} if and only if for all $i\in\{1,...,n\}$ it holds that $f_i(x) = 0$ or $\nabla f_i(x)=0$.
	\end{prop}
	
	\begin{proof}
		If $f_i(x)=0$ or $\nabla f_i(x)=0$ holds for all $i\in\{1,...,n\}$, then $(x,x^*)$ is a fixed point by definition of the steps. 
		Next, we assume that $x\in\dom\partial\varphi$ is a fixed point of Algorithm~\ref{alg:NBK} and $\nabla f_i(x)\neq 0$. First, assume that condition~\eqref{eqn:Condition_hyperplane_nonempty_intersection} is not fulfilled, then the update for $x^*$ shows that $t_{k,\sigma}=0$, since $\nabla f_i(x)\neq 0$, and hence, $f_i(x)=0$. Finally, we assume that condition~\eqref{eqn:Condition_hyperplane_nonempty_intersection} holds. Then, from Proposition~\ref{prop:t_min_problem_abstract_hyperplane} we know that Algorithm~\ref{alg:NBK} computes the Bregman projection $x = \Pi^{x^*}_{\varphi,H}(x)$ with 
		\[ H = \big\{ y\in\RR^d: f_i(x) + \langle \nabla f_i(x),y-x\rangle = 0\big\}. \]
		But this means that $x\in H$ and hence, $f_i(x) = 0$ holds also in this case. 
	\end{proof}

	The following fact will be useful in the convergence analysis. It shows that Algorithm~\ref{alg:NBK} performs a mirror descent step whenever $f_i(x)>0$, and a mirror ascent step whenever $f_i(x)<0$.

	\begin{lem}
		\label{lem:sign_of_stepsize}
		Consider the $k$th iterate $x_k$ of Algorithm~\ref{alg:NBK} and consider the case that $f_{i_k}(x_k)\neq 0$ and $\nabla f_{i_k}(x_k)\neq 0$. Let Assumption~\ref{as:Standard_assumptions} hold.
		Then, the step size $t_k$ in Algorithm~\ref{alg:NBK} fulfills
		\begin{align*} 
			\mathrm{sign}(t_k) = \mathrm{sign}(f_{i_k}(x_k)). 
		\end{align*}
	\end{lem}
	
	\begin{proof}
		If condition~\eqref{eqn:Condition_hyperplane_nonempty_intersection} is not fulfilled, the assertion is clear by definition of the step size. Next, we assume that~\eqref{eqn:Condition_hyperplane_nonempty_intersection} holds.
		Then, the function
		\begin{align} 
			\label{eqn:g_function}
			g_{i_k,x_k^*}(t) = \varphi^*(x_k^*-t\nabla f_{i_k}(x_k)) + t\big(\langle \nabla f_{i_k}(x_k),x_k\rangle - f_{i_k}(x_k)\big)
		\end{align}
		is minimized by $t_{k,\varphi}$. (Note that the expression is indeed fully determined by $i_k$ and $x_k^*$ by the fact that $x_k=\nabla\varphi^*(x_k^*$)). We compute  
		\begin{align} 
			\label{eqn:g_derivative_at_zero}
			g_{i_k,x_k^*}'(0) = -\langle \nabla f_{i_k}(x_k), \ \nabla \varphi^*(x_k^*)\rangle + \langle \nabla f_{i_k}(x_k),x_k\rangle - f_{i_k}(x_k) = -f_{i_k}(x_k). 
		\end{align}
		Since $g_{i_k,x_k^*}$ is convex, its derivative is monotonically increasing and it vanishes at~$t_{k,\varphi}$. Since it holds $f_{i_k}(x_k)\neq 0$ by assumption, we conclude that $t_{k,\varphi}$ and $f_{i_k}(x_k)$ have the same sign.
	\end{proof}

	To exploit strong convexity and smoothness, we will use the following. 
	
	\begin{lem}
		\label{lem:BasicsConvexAnalysis_strongly_convex}
		If $\varphi\colon\RR^d\to\RR$ is proper, convex and lower semicontinuous, then the following statements are equivalent:  
		\begin{enumerate}[(i)]
			\item $\varphi$ is $\sigma$-strongly convex w.r.t. $\|\cdot\|$. 
			\item For all $x,y\in\RR^d$ and $x^*\in\partial\varphi(x)$, $y^*\in\partial\varphi(y)$,
			\[ \langle x^*-y^*, x-y\rangle \geq \sigma\|x-y\|^2. \]
			\item The function $\varphi^*$ is $\tfrac{1}{\sigma}$-smooth w.r.t. $\|\cdot\|_*$.
		\end{enumerate}
	\end{lem}
	
	\begin{proof}
		See \cite[Corollary 3.5.11 and Remark 3.5.3]{Z02}.
	\end{proof}

	\begin{lem}
		\label{lem:BasicsConvexAnalysis_convex_Lsmooth}
		If $\varphi\colon\RR^d\to\RR$ is convex and lower semicontinuous, then the following statements are equivalent: 
		\begin{enumerate}[(i)]
			\item $\varphi$ is $L$-smooth w.r.t. a norm $\|\cdot\|$, 
			\item $\varphi(y) \leq \varphi(x) + \langle \nabla\varphi(x),y-x\rangle + \frac{L}{2} \|x-y\|^2$ for all $x,y\in\RR^d$,
			\item $\langle \nabla\varphi(y) - \nabla\varphi(x),y-x\rangle \leq L\|x-y\|^2$ for all $x,y\in\RR^d$.
		\end{enumerate}
	\end{lem}
	
	\begin{proof}
		See \cite[Corollary 3.5.11 and Remark 3.5.3]{Z02}.
	\end{proof}

	\subsection{Convergence for nonnegative star-convex functions}
	\label{subsec:star-convex}
	
	In this subsection we assume in addition that the functions $f_i$ are either nonnegative and star-convex or affine. 
	
	\begin{defi}[\cite{NP06}]
		\label{as:Nonneg_starconvex_nonpos_starconcave}
		Let $f\colon D\to\RR$ be differentiable. We say that $f$ is called star-convex, if the set $\argmin f$ is nonempty and for all $x\in D$ and $\hat x\in\argmin f$ it holds that 
		\begin{align*}
			f(x) + \langle \nabla f(x), \hat x - x \rangle \leq f(\hat x).
		\end{align*}
		Moreover, we call $f$ strictly star-convex, if the above inequality is strict, and $\mu$- strongly star-convex relative to $\varphi$, if for all $x\in D$, $x^*\in\partial\varphi(x)$ and $\hat x\in\argmin f$ it holds that 
		\begin{align*}
			f(x) + \langle \nabla f(x), \hat x - x \rangle + \mu D_\varphi^{x^*}(x,\hat x) \leq f(\hat x).
		\end{align*}
	\end{defi}
	
	We recall that the first assumption of nonnegativity and star-convexity covers two settings: 
	\begin{itemize}
		\item Minimizing a sum-of-terms \begin{equation}\label{eq:sum-of-terms}
			\min \frac{1}{n} \sum_{i=1}^n f_i(x)\qquad \text{s.t. } x\in C,
		\end{equation}  
		under the interpolation assumption \begin{equation} \label{eq:interpolation}
			\text{ }\exists x: \quad x\in\bigcap_{i=1}^n \argmin f_i|_C, 
		\end{equation} 
		where $f_i$ is a star-convex function with known optimal value $\hat f_i$ on $C$. Under the interpolation assumption every point in the intersection on the right hand side of~\eqref{eq:interpolation} 
		is a solution to~\eqref{eq:sum-of-terms}.
		Furthermore, we will construct a sequence which converges to this intersection point by applying Algorithm~\ref{alg:NBK} to the nonnegative function $\tilde f$ where $\tilde f_i = f_i|_C - \hat f_i$. When  $n=1$, we cover the setting of mirror descent for the problem
		\begin{align}
			\label{eqn:problem_n=1}
			\min f(x) \quad \text{s.t. } x\in C
		\end{align}
		with known optimal value $\hat f$.
		\item Systems of nonlinear equations \[ f(x) = 0 \qquad \text{s.t. } x\in C \]
		with star-convex component functions $f_i$, where we apply Algorithm~\ref{alg:NBK} to $f_i^+  = \max(f_i,0)$. Note that $f_i^+$ is not differentiable only at points $x$ with $f_i(x)=0$, which is anyway checked during the method.
	\end{itemize}

	Precisely, we will use the following assumption.
	\begin{as}
		\label{as:Group_of_assumptions_on_f_nonnegative_and_star_convex_part}
		For each $f_i$ one of the following conditions is fulfilled:
		\begin{enumerate}[(i)]
			\item $f_i$ is nonnegative and star-convex and it holds that $f_i^{-1}(0) \cap \dom\partial\varphi\neq\emptyset$,
			\item $f_i$ is nonnegative and strictly star-convex or 
			\item $f_i$ is affine.
		\end{enumerate}
	\end{as}

	The first theorem states that Algorithm~\ref{alg:NBK} always computes nonrelaxed Bregman projections under Assumption~\ref{as:Group_of_assumptions_on_f_nonnegative_and_star_convex_part} outside of the fixed points.

	\begin{thm}
		\label{thm:Definedness_nonnegative_starconvex}
		Let $\big(x_k,x_k^*\big)$ be given by Algorithm~\ref{alg:NBK} and consider the case that $f_{i_k}(x)\neq 0$ and $\nabla f_{i_k}(x)\neq 0$. Let Assumptions~\ref{as:Standard_assumptions}-\ref{as:Group_of_assumptions_on_f_nonnegative_and_star_convex_part} hold true. Then, the hyperplane $H_k$ separates $x_k$ and $f_{i_k}^{-1}(0)$, the condition 
		\begin{align*}
			H_k \cap \dom\partial\varphi\neq\emptyset
		\end{align*}
		holds and the Bregman projection of $x_k$ onto $H_k$ is defined, namely it holds that \[ x_{k+1}=\Pi_{\varphi,H_k}^{x_k^*}(x_k). \]
		In particular, Algorithm~\ref{alg:NBK} always chooses the step size $t_k=t_{k,\varphi}$ from~\eqref{eqn:BregProj_stepsize}.
	\end{thm} 
	
	\begin{proof}
		For $x\in D$ we define the affine function
		\begin{align*}
			\ell_x(y) := f_{i_k}(x) + \langle \nabla f_{i_k}(x), y-x\rangle.
		\end{align*}
		We consider the cases (i) and (ii) from Assumption~\ref{as:Group_of_assumptions_on_f_nonnegative_and_star_convex_part} first. Here we have that $\ell_{x_k}(x_k)=f_{i_k}(x_k)> 0$ and for all $\hat x\in f_{i_k}^{-1}(0)$, star-convexity of $f_{i_k}$ shows that $\ell_{x_k}(\hat x) \leq 0$. This means that the hyperplane~$H_k$ separates $x_k$ and $f_{i_k}^{-1}(0)$. By the intermediate value theorem there exists $x^{\lambda} = \lambda x_k + (1-\lambda)\hat x$ for some $\lambda\in [0,1[$ such that $\ell_{x_k}(x^{\lambda})=0$. Now let us assume that Assumption~\ref{as:Group_of_assumptions_on_f_nonnegative_and_star_convex_part}(i) holds, so we can choose $\hat x\in\dom\partial\varphi$ with $f_{i_k}(\hat x)=0$. By Assumption~\ref{as:Standard_assumptions}(iii) it holds that $\dom\partial\varphi=\ri\dom\varphi$, which is a convex set and hence, $x^{\lambda}\in\dom\partial\varphi$. This proves that the claimed condition~\eqref{eqn:Condition_hyperplane_nonempty_intersection} is fulfilled and the update in Algorithm~\ref{alg:NBK} computes a Bregman projection onto $H_k$ by Proposition~\ref{prop:t_min_problem_abstract_hyperplane}. In case of Assumption~\ref{as:Group_of_assumptions_on_f_nonnegative_and_star_convex_part}(ii) we have that $\ell_{x_k}(\hat x)<0$ and therefore it even holds that $\lambda\in ]0,1[$. Assumption~\ref{as:Standard_assumptions}(iii) guarantees that $\hat x \in \overline{\dom\varphi}$ and $x_k\in\dom\partial\varphi = \ri\dom\varphi$. Hence, we have $x^{\lambda}\in \ri\dom\varphi$ by \cite[Theorem 6.1]{Rockafellar70}, which again implies that $x^{\lambda}\in\dom\partial\varphi$ by Assumption~\ref{as:Standard_assumptions}(iii). This proves that condition~\eqref{eqn:Condition_hyperplane_nonempty_intersection} is fulfilled in this case, too. \\
		Under Assumption~\ref{as:Group_of_assumptions_on_f_nonnegative_and_star_convex_part}(iii) it holds that $\ell_x = f_{i_k}$ for all $x\in D$. This already implies that $H_k$ separates $x_k$ and $f_{i_k}^{-1}(0)$. Condition~\eqref{eqn:Condition_hyperplane_nonempty_intersection} is fulfilled by the assumption that $f_{i_k}^{-1}(0)\neq\emptyset$ and so, the update computes the claimed Bregman projection by Proposition~\ref{prop:t_min_problem_abstract_hyperplane} also in this case.
	\end{proof}
	
	As an immediate consequence, we see that Algorithm~\ref{alg:NBK} is stable in terms of Bregman distance. 
	
	\begin{cor}
		If $\hat x\in S$ is a solution to~\eqref{eqn:problem} and the assumptions from Theorem~\ref{thm:Definedness_nonnegative_starconvex} hold true,
		then it holds that
		\begin{align*}
			D_\varphi^{x_{k+1}^*}(x_{k+1},\hat x) \leq D_\varphi^{x_k^*}(x_k,\hat x) - D_\varphi^{x_k^*}(x_k,x_{k+1}).
		\end{align*}
	\end{cor}
	\begin{proof} 
		By Theorem~\ref{thm:Definedness_nonnegative_starconvex}, $x_{k+1}$ is the Bregman projection of $x_k$ onto $H_{k}$ with respect to $x_k^*$. If $f_{i_k}(x_k)=0$, then $x_k$ is a fixed point by Proposition~\ref{prop:Fixed_points}and the statement holds trivially. Next, assume that $f_{i_k}(x_k)>0$. Then by Lemma~\ref{lem:sign_of_stepsize}, we have that $t_k > 0$. As $x_{k+1}\in H$, we have that $\langle \nabla f_{i_k}(x_k), x_{k+1}-x_k\rangle=0$. We conclude for all $y \in H_k^\leq$ that 
		\begin{align*}
			\langle x_{k+1}^*-x_k^*, x_{k+1}-y\rangle 
			&= -t_k\langle \nabla f_{i_k}(x_k), x_{k+1}-y\rangle\\ 
			&=  t_k\langle \nabla f_{i_k}(x_k), y-x_k\rangle
			-t_k\langle \nabla f_{i_k}(x_k), x_{k+1}-x_k\rangle \\
			&\leq -t_k f_{i_k}(x_k) \\
			&\leq 0,
		\end{align*}
		which by Lemma~\ref{lem:BregmanProj_VarInequality} shows that $x_{k+1} = \Pi_{\varphi,H^\leq_k}^{x_k^*}(x_k)$. As $\hat x\in H^\leq_k$, the claim follows from 
		Lemma~\ref{lem:BregmanProj_VarInequality}(ii). An analoguous argument shows the claim in the case $f_{i_k}(x_k)<0$.
	\end{proof}

	Next, we prove that the exact Bregman projection moves the iterates closer to solutions of~\eqref{eqn:problem} than the relaxed projections, where the distance is in the sense of the used Bregman distance (see Theorem~\ref{thm:Comparison_dist_to_sol_Bregman_mSPS}). To that end, for $(x_k,x_k^*)$ from Algorithm~\ref{alg:NBK} we define an update with variable step size 
	\begin{align}
		\label{eqn:Update_with_variable_stepsize}
		x_{k+1}^*(t) = x_k^* - t\nabla f_{i_k}(x_k), \qquad x_{k+1}(t) = \nabla\varphi^*(x_{k+1}^*(t)), \quad t\in\RR.
	\end{align}
	
	\begin{lem}
		\label{lem:Basic_Dphi_descent_estimate}
		Let $\big(x_k,x_k^*\big)$ be given by Algorithm~\ref{alg:NBK} and consider $(x_{k+1}(t),x_{k+1}^*(t))$ from~\eqref{eqn:Update_with_variable_stepsize} for some $t\in\RR$. 
		Let Assumption~\ref{as:Standard_assumptions} hold true. Then, for all $x\in \RR^d$ it holds that
		\begin{align*}
			D_\varphi^{x_{k+1}^*(t)}(x_{k+1}(t),x) 
			&= \varphi^*(x_k^*-t\nabla f_{i_k}(x_k)) + t\beta_k + t\langle\nabla f_{i_k}(x_k), x-x_k\rangle  \\
			& \hspace{1cm} + t f_{i_k}(x_k) - \langle x_k^*,x\rangle + \varphi(x). 
		\end{align*}
	\end{lem}
	\begin{proof}
		Rewriting the Bregman distance as in~\eqref{eqn:BregDist_with_conjugate} shows that
		\begin{align*}
			D_\varphi^{x_{k+1}^*(t)}(x_{k+1}(t),x) 
			&= \varphi^*(x_{k+1}^*(t)) - \langle x_{k+1}^*(t),x\rangle + \varphi(x) \\
			& = \varphi^*(x_k^*-t\nabla f_{i_k}(x_k))  + t\langle \nabla f_{i_k}(x_k),x\rangle - \langle x_k^*,x\rangle + \varphi(x) \\
			&= \varphi^*(x_k^*-t\nabla f_{i_k}(x_k)) + t\beta_k + t\langle\nabla f_{i_k}(x_k), x-x_k\rangle  \\
			&\hspace{1cm} + t f_{i_k}(x_k) - \langle x_k^*,x\rangle + \varphi(x). 
		\end{align*}
	\end{proof}

	\begin{prop}
		\label{prop:Comparison_dist_to_arbitrary_points_Bregman_SPS}
		Let $\big(x_k,x_k^*\big)$ and $t_k$ be given by Algorithm~\ref{alg:NBK} and consider $\big(x_{k+1}(t),x_{k+1}^*(t)\big)$ from~\eqref{eqn:Update_with_variable_stepsize} for some $t\in\RR$. 
		Let Assumption~\ref{as:Standard_assumptions} hold true. Then, for all $x\in \RR^d$ it holds that
		\begin{align*}
			D_\varphi^{x_{k+1}^*}(x_{k+1},x) \leq D_\varphi^{x_{k+1}^*(t)}(x_{k+1}(t),x) + (t_k-t) \cdot \big( f_{i_k}(x_k) + \langle \nabla f_{i_k}(x_k), x - x_k\rangle \big).
		\end{align*}
	\end{prop}

	\begin{proof}
		Note that $t_k$ equals $t_{k,\varphi}$ from~\eqref{eqn:BregProj_stepsize}, since condition~\eqref{eqn:Condition_hyperplane_nonempty_intersection} is fulfilled by Theorem~\ref{thm:Definedness_nonnegative_starconvex}. 
		Hence, the optimality property~\eqref{eqn:BregProj_stepsize} shows that for any $t$ we have that 
		\[ \varphi^*(x_k^*-t_k\nabla f_{i_k}(x_k)) + t_k\beta_k \leq \varphi^*(x_k^*-t\nabla f_{i_k}(x_k)) + t\beta_k. \] 
		Lemma~\ref{lem:Basic_Dphi_descent_estimate} then shows that for any $x$ it holds that 
		\begin{align*} D_\varphi^{x_{k+1}^*}(x_{k+1},x) 
			&\leq \varphi^*(x_k^*-t\nabla f_{i_k}(x_k)) + t\beta_k + t_k\langle\nabla f_{i_k}(x_k), x-x_k\rangle + t_k f_{i_k}(x_k) 
			\\
			& \hspace{1cm} - \langle x_k^*,x\rangle + \varphi(x). 
		\end{align*}
		We use the definitions of $\beta_k$, $x_{k+1}(t)$ and $x_{k+1}^*(t)$ and get
		\begin{align*}
			D_\varphi^{x_{k+1}^*}(x_{k+1},x) 
			&\leq \varphi^*(x_k^*-t\nabla f_{i_k}(x_k)) + t\beta_k + t_k\langle\nabla f_{i_k}(x_k), x-x_k\rangle + t_k f_{i_k}(x_k)  \\
			& \hspace{1cm} - \langle x_k^*,x\rangle + \varphi(x) 
			\\
			&= \varphi^*(x_k^*-t\nabla f_{i_k}(x_k)) + t\langle \nabla f_{i_k}(x_k), x\rangle - \langle x_k^*,x\rangle + \varphi(x)  \\
			& \hspace{1cm} + (t_k-t) \cdot \big( f_{i_k}(x_k) + \langle \nabla f_{i_k}(x_k), x - x_k\rangle \big)  \\
			&= D_\varphi^{x_{k+1}^*(t)}(x_{k+1}(t),x) + (t_k-t) \cdot \big( f_{i_k}(x_k) + \langle \nabla f_{i_k}(x_k), x - x_k\rangle \big).
		\end{align*}
	\end{proof}
	
	In order to draw a conclusion from Proposition~\ref{prop:Comparison_dist_to_arbitrary_points_Bregman_SPS}, we relate the step sizes $t_{k,\varphi}$ from~\eqref{eqn:BregProj_stepsize} and $t_{k,\sigma}$ from~\eqref{eqn:mSPS_like_stepsize}. We already know by Lemma~\ref{lem:sign_of_stepsize} that both step sizes have the same sign. The next lemma gives upper and lower bounds for $t_{k,\varphi}$ with respect to $t_{k,\sigma}$ under additional assumptions on~$\varphi$.
	
	\begin{lem}
		\label{lem:Bounds_Bregman_SPS}
		Let $\big(x_k,x_k^*\big)$ be the iterates from Algorithm~\ref{alg:NBK} and let Assumption~\ref{as:Standard_assumptions} hold true. We consider $t_{k,\varphi}$ and $t_{k,\sigma}$ from~\eqref{eqn:BregProj_stepsize} and~\eqref{eqn:mSPS_like_stepsize} and the function $g_{i_k,x_k^*}$ from~\eqref{eqn:g_function}. 
		\begin{enumerate}
			\item[\textup{(i)}] 	
			If $\varphi$ is $\sigma$-strongly convex w.r.t. $\|\cdot\|$, then $g_{i_k,x_k^*}$ is $\frac{\|\nabla f_{i_k}(x_k)\|_*^2}{\sigma}$-smooth and 
			\begin{align} 
				\label{eqn:t_k_varphi_lower_bound}
				|t_{k,\varphi}| \geq \sigma \frac{ |f_{i_k}(x_k)| }{\|\nabla f_{i_k}(x_k)\|_*^2} = |t_{k,\sigma}|.
			\end{align}
			\item[\textup{(ii)}]
			If $\varphi$ is $M$-smooth w.r.t. $\|\cdot\|$, then $g_{i_k,x_k^*}$ is $\frac{\|\nabla f_{i_k}(x_k)\|_*^2}{M}$-strongly convex and 
			\begin{align} 
				\label{eqn:t_k_varphi_upper_bound}
				|t_{k,\varphi}| \leq M\frac{|f_{i_k}(x_k)|}{\|\nabla f_{i_k}(x_k)\|_*^2} = \frac{M}{\sigma} \cdot |t_{k,\sigma}|. 
			\end{align}
		\end{enumerate} 
	\end{lem}
	
	\begin{proof}
		For $s,t\in\RR$ with $s<t$ it holds that
		\begin{align}
			g_{i_k,x_k^*}'(t)-g_{i_k,x_k^*}'(s) 
			&= \langle 
			\nabla \varphi^*(x_k^*-t\nabla f_{i_k}(x_k)) 
			-\nabla \varphi^*(x_k^*-s\nabla f_{i_k}(x_k)), -\nabla f_{i_k}(x_k)
			\rangle \nonumber \\
			&= \frac{1}{(t-s)} \big\langle
			\nabla \varphi^*(x_k^*-t\nabla f_{i_k}(x_k)) -\nabla \varphi^*(x_k^*-s\nabla f_{i_k}(x_k)), \nonumber \\
			& \hspace{2.65cm} x_k^*-t\nabla f_{i_k}(x_k) - (x_k^*-s\nabla f_{i_k}(x_k))
			\big\rangle. \label{eqn:g_derivative_diff}
		\end{align}
		\begin{enumerate}
			\item[\textup{(i)}] If $\varphi$ is $\sigma$-strongly convex w.r.t. $\|\cdot\|$, then $\varphi^*$ is $\frac1\sigma$-smooth w.r.t. $\|\cdot\|_*$ by Lemma~\ref{lem:BasicsConvexAnalysis_strongly_convex}(iv). Hence, by Lemma~\ref{lem:BasicsConvexAnalysis_convex_Lsmooth}(iii) we can estimate 
			\begin{align*}
				0 \leq g_{i_k,x_k^*}'(t)-g_{i_k,x_k^*}'(s) \leq \frac{\|(t-s)\nabla f_{i_k}(x_k)\|_*^2}{\sigma\cdot(t-s)}  = \frac{\|\nabla f_{i_k}(x_k)\|_*^2}{\sigma} \cdot (t-s),
			\end{align*}
			which proves by the same lemma that $g_{i_k,x_k^*}$ is $\frac{\|\nabla f_{i_k}(x_k)\|_*^2}{\sigma}$-smooth. Hence, \eqref{eqn:t_k_varphi_lower_bound} follows by choosing $t=\max(t_{k,\varphi},0)$, $s=\min(t_{k,\varphi},0)$ and inserting~\eqref{eqn:g_derivative_at_zero}. 
			
			\item[(ii)] Here, by Lemma~\ref{lem:BasicsConvexAnalysis_strongly_convex} and the Fenchel-Moreau-identity $\varphi=\varphi^{**}$, the function $\varphi^*$ is $\frac1M$-strongly convex w.r.t. $\|\cdot\|_*$. Using Lemma~\ref{lem:BasicsConvexAnalysis_strongly_convex}(ii) we can estimate
			\begin{align*}
				g_{i_k,x_k^*}'(t)-g_{i_k,x_k^*}'(s) \geq \frac{\|(t-s)\nabla f_{i_k}(x_k)\|_*^2}{M\cdot(t-s)}  = \frac{\|\nabla f_{i_k}(x)\|_*^2}{M} \cdot (t-s)
			\end{align*}
			which shows that $g_{i_k,x_k^*}$ is $\frac{\|\nabla f_{i_k}(x_k)\|_*^2}{M}$-strongly convex. Inequality~\eqref{eqn:t_k_varphi_upper_bound} then follows as in \textup{(i)}.
		\end{enumerate}
	\end{proof}

	\begin{thm}
		\label{thm:Comparison_dist_to_sol_Bregman_mSPS}
		Let $\big(x_k,x_k^*\big)$ and $t_k$ be given by Algorithm~\ref{alg:NBK}. Let $t\in [0,t_k]$ and let $\big(x_{k+1}(t),x_{k+1}^*(t)\big)$ be as in~\eqref{eqn:Update_with_variable_stepsize}.
		Let Assumptions~\ref{as:Standard_assumptions}-\ref{as:Group_of_assumptions_on_f_nonnegative_and_star_convex_part} hold true and assume that $f_{i_k}(x_k)>0$. 
		Then for every solution $\hat x\in S$ it holds that 
		\begin{align*}
			D_\varphi^{x_{k+1}^*}(x_{k+1}, \hat x) \leq D_\varphi^{x_{k+1}^*(t)}(x_{k+1}(t), \hat x).
		\end{align*}
		If $\varphi$ is $\sigma$-strongly convex, the inequality holds as well for $t=t_{k,\sigma}$.
	\end{thm}
	
	\begin{proof}
		We recall that $t_k$ equals $t_{k,\varphi}$ from~\eqref{eqn:BregProj_stepsize}, since condition~\eqref{eqn:Condition_hyperplane_nonempty_intersection} is fulfilled by Theorem~\ref{thm:Definedness_nonnegative_starconvex}. 
		By Lemma~\ref{lem:sign_of_stepsize} we have that $t_k>0$. Since $f_{i_k}$ is star-convex and $\hat x\in S$, it holds that 
		\begin{align*}
			f_{i_k}(x_k) + \langle \nabla f_{i_k}(x_k),\hat x-x_k\rangle \leq 0.
		\end{align*}
		The statement now follows from Proposition~\ref{prop:Comparison_dist_to_arbitrary_points_Bregman_SPS}. The theorem applies in particular to $t=t_{k,\sigma}$ if $\varphi$ is $\sigma$-strongly convex, as Lemma~\ref{lem:Bounds_Bregman_SPS}(i) ensures that $0<t_{k,\sigma}\leq t_{k,\varphi}$.
	\end{proof}
	
	For mirror descent or stochastic mirror descent under interpolation, Theorem~\ref{thm:Comparison_dist_to_sol_Bregman_mSPS} tells that a choice of a smaller step size than $t_{k,\varphi}$ results in a larger distance to solutions~$\hat x$ of problem~\eqref{eqn:problem} in Bregman distance.

	The following lemma is the key element of our convergence analysis. 
	
	\begin{lem}
		\label{lem:Dphi_descent_estimate}
		Let $\big(x_k,x_k^*\big)$ be the iterates of either Algorithm~\ref{alg:NBK} or Algorithm~\ref{alg:NBK_relaxed}. Let Assumptions~\ref{as:Standard_assumptions}-\ref{as:Group_of_assumptions_on_f_nonnegative_and_star_convex_part} hold true and assume that $\varphi$ is $\sigma$-strongly convex w.r.t. a norm $\|\cdot\|$. 
		Then for every solution $\hat x\in S$ it holds that 
		\begin{align}
			\label{eqn:Dphi_descent_estimate_leq}
			D_\varphi^{x_{k+1}^*}(x_{k+1},\hat x) 
			\leq  D_\varphi^{x_k^*}(x_k,\hat x) - \frac{\sigma}{2} \frac{\big(f_{i_k}(x_k)\big)^2}{\|\nabla f_{i_k}(x_k)\|_*^2}. 
		\end{align}
	\end{lem}
	\begin{proof}
		We bound the right-hand side in Lemma~\ref{lem:Basic_Dphi_descent_estimate} from above for $t\in\{t_{k,\varphi},t_{k,\sigma}\}$. As $\varphi$ is $\sigma$-strongly convex, $\varphi^*$ is $\frac1\sigma$-smooth by Lemma~\ref{lem:BasicsConvexAnalysis_strongly_convex}(iii). Hence, by Lemma~\ref{lem:BasicsConvexAnalysis_convex_Lsmooth}(ii), for all $t\in\RR$ we can estimate that
		\begin{align*}
			\varphi^*(&x_k^*-t\nabla f_{i_k}(x_k)) + t\beta_k \\
			&=\varphi^*(x_k^*-t\nabla f_{i_k}(x_k)) + t\big(\langle \nabla f_{i_k}(x_k), x_k\rangle - f_{i_k}(x_k)\big) \\
			&\leq \varphi^*(x_k^*) - t\langle \nabla \varphi^*(x_k^*),\nabla f_{i_k}(x_k)\rangle + \frac{1}{2\sigma} t^2\|\nabla f_{i_k}(x_k)\|_*^2 \\
			& \hspace{0.5cm} + t\big(\langle \nabla f_{i_k}(x_k), x_k\rangle - f_{i_k}(x_k)\big) \\
			&= \varphi^*(x_k^*) - t f_{i_k}(x_k) + \frac{1}{2\sigma} t^2\|\nabla f_{i_k}(x_k)\|_*^2.
		\end{align*}
		Minimizing the right hand side over $t\in\RR$ gives $\hat t = \sigma \frac{ f_{i_k}(x_k)}{\|\nabla f_{i_k}(x_k)\|_*^2} = t_{k,\sigma}$ and 
		\[ \varphi^*(x_k^*) - \hat t f_{i_k}(x_k) + \frac{1}{2\sigma} \hat t^2\|\nabla f_{i_k}(x_k)\|_*^2 = \varphi^*(x_k^*) - \frac{\sigma}{2} \frac{\big(f_{i_k}(x_k)\big)^2}{\|\nabla f_{i_k}(x_k)\|_*^2}. \]
		By optimality of $t_{k,\varphi}$ we have that 
		\begin{align*}
			\varphi^*(&x_k^*-t_{k,\varphi}\nabla f_{i_k}(x_k)) + t_{k,\varphi}\beta_k 
			\leq \varphi^*(x_k^*-t\nabla f_{i_k}(x_k)) + t\beta_k
		\end{align*}
		for all $t\in\RR$. Hence, we have shown that  
		\begin{align} 
			\label{eqn:g_estimate}
			\varphi^*(x_k^*-t\nabla f_{i_k}(x_k)) + t\beta_k \leq \varphi^*(x_k^*) - \frac{\sigma}{2} \frac{\big(f_{i_k}(x_k)\big)^2}{\|\nabla f_{i_k}(x_k)\|_*^2} \end{align}
		holds for $t\in\{t_{k,\sigma},t_{k,\varphi}\}$. If Assumption~\ref{as:Group_of_assumptions_on_f_nonnegative_and_star_convex_part}(i) or~\ref{as:Group_of_assumptions_on_f_nonnegative_and_star_convex_part}(ii) are fulfilled, Lemma~\ref{lem:sign_of_stepsize} and star-convexity of $f_{i_k}$ show that
		\begin{align} 
			\label{eqn:Estimate_leq_convexity_estimate}
			t_k\big(f_{i_k}(x_k)+\langle \nabla f_{i_k}(x_k), \hat x-x_k\rangle\big) 
			\leq 0. 
		\end{align}
		Under Assumption~\ref{as:Group_of_assumptions_on_f_nonnegative_and_star_convex_part}(iii), we have equality in~\eqref{eqn:Estimate_leq_convexity_estimate}. 
		Inserting this inequality into Lemma~\ref{lem:Basic_Dphi_descent_estimate}, we obtain the claimed bound 
		\begin{align*}
			D_\varphi^{x_{k+1}^*}(x_{k+1},\hat x)
			&\leq \varphi^*(x_k^*) - \frac{\sigma}{2} \frac{\big(f_{i_k}(x_k)\big)^2}{\|\nabla f_{i_k}(x_k)\|_*^2} - \langle x_k^*,\hat x\rangle + \varphi(\hat x) \\
			&= D_\varphi^{x_k^*}(x_k,\hat x) - \frac{\sigma}{2} \frac{\big(f_{i_k}(x_k)\big)^2}{\|\nabla f_{i_k}(x_k)\|_*^2},
		\end{align*}
		where we used~\eqref{eqn:BregDist_with_conjugate} in the last step. \\
	\end{proof}

	We can now establish almost sure (a.s.) convergence of Algorithm~\ref{alg:NBK}. The expectations are always taken with respect to the random choice of the indices. Sometimes we also take conditional expectations conditioned on choices of indices in previous iterations, which we will indicate explicitly.

	\begin{thm} 
		\label{thm:Convergence_nonnegative_convex_C1_leq}
		Let Assumptions~\ref{as:Standard_assumptions}-\ref{as:Group_of_assumptions_on_f_nonnegative_and_star_convex_part} hold true and assume that $\varphi$ is $\sigma$-strongly convex w.r.t. a norm $\|\cdot\|$. Then it holds that
		\[ \EE\big[\sum_{i=1}^n p_i \big(f_i(x_k)\big)^2 \big]\to 0 \quad \text{as } k\to\infty \]
		and we have the rate 
		\begin{align*} \EE\big[\min_{l=1,...,k} \sum_{i=1}^n p_i \big(f_i(x_l)\big)^2\big] 
			\leq \frac{c}{\sigma k} \end{align*}
		with some constant $c>0$. Moreover, the iterates $x_k$ of Algorithm~\ref{alg:NBK} converge a.s. to a random variable whose image is contained in the solution set~$S$. 
	\end{thm}
	
	\begin{proof}
		By $\sigma$-strong convexity of $\varphi$ and Lemma~\ref{lem:Dphi_descent_estimate}, we have that
		\[ \frac{\sigma}{2}\|x_k-\hat x\|^2 \leq D_\varphi^{x_k^*}(x_k,\hat x) \leq D_\varphi^{x_0^*}(x_0,\hat x) \]
		holds for all $k\in\NN$ and $\hat x\in S$. Hence, the sequence $x_k$ is bounded and we have 
		\begin{align*}
			\|\nabla f_{i_k}(x_k)\|_*^2 \leq M 
		\end{align*}
		with some constant $M>0$. Inserting this into \eqref{eqn:Dphi_descent_estimate_leq} gives that for all $l\in\NN$ it holds
		\begin{align*}
			D_\varphi^{x_{l+1}^*}(x_{l+1},\hat x) \big) 
			&\leq D_\varphi^{x_l^*}(x_l,\hat x) - \frac{\sigma}{2M} \big(f_{i_l}(x_l)\big)^2.
		\end{align*}
		Taking conditional expectation w.r.t. $i_0,...,i_{l-1}$, we obtain
		\begin{align}
			\EE\big[D_\varphi^{x_{l+1}^*}(x_{l+1},\hat x)  \mid i_0,...,i_{l-1}\big]  
			&\leq D_\varphi^{x_l^*}(x_l,\hat x)- \frac{\sigma}{2M} \sum_{i=1}^n p_i\big(f_{i_l}(x_l)\big)^2.
			\label{eqn:Descent_estimate_in_thm_proof_leq} 
		\end{align}
		By rearranging and using the tower property of conditional expectation, we conclude that
		\begin{align*}
			\EE\big[ \sum_{i=1}^n p_i\big(f_{i_l}(x_l)\big)^2 \big] \leq
			\frac{2M}{\sigma} 	\big( \EE\big[ D_\varphi^{x_l^*}(x_l,\hat x) \big] - \EE\big[ D_\varphi^{x_{l+1}^*}(x_{l+1},\hat x) \big] \big)
		\end{align*}
		The convergence rate now follows with $c=2\cdot M\cdot D_{\varphi}^{x_0^*}(x_0,\hat x)$ for any $\hat x\in S$ by averaging over $l=0,...,k$ and telescoping. 
		
		Next, we prove the a.s. iterate convergence. Using \eqref{eqn:Descent_estimate_in_thm_proof_leq} gives that
		\begin{align*}
			\EE\big[D_\varphi^{x_{k+1}^*}(x_{k+1},\hat x) \mid i_0,...,i_{k-1}\big]  
			\leq 
			D_\varphi^{x_k^*}(x_{k},\hat x) - \frac{\sigma \cdot \min_i p_i}{2M} \cdot \|f(x_k)\|_2^2. 
		\end{align*}
		The Robbins-Siegmund Lemma~\cite{Robbins71} proves that $f(x_k)\to 0$ holds with probability~$1$. 
		Along any sample path in $\{f(x_k)\to 0\}$, due to boundedness of the sequence $x_k$, there exists a subsequence $x_{k_l}$ converging to some point $x$. By continuity, we have that $f(x)=0$ and hence, $x\in S$. Due to Assumption~\ref{as:Standard_assumptions}(iv), it holds that $D_\varphi^{x_{k_l}^*}(x_{k_l}, x)\to 0$ and since $D_\varphi^{x_k^*}(x_k,\hat x)$ is a decreasing sequence in $k$ for $\hat x\in S$ by Lemma \ref{lem:Dphi_descent_estimate}, we conclude that $D_\varphi^{x_{k}^*}(x_{k}, x)\to 0.$ Finally, strong convexity of $\varphi$ implies that $x_k\to x$.
	\end{proof}
	If the functions $f_i$ have Lipschitz continuous gradient, we can derive a sublinear rate for the $\ell_1$-kind loss $\EE\big[\min_{l=1,...,k} \sum_{i=1}^n p_i f_i(x_l)\big]$, which coincides with the rate in~\cite[Theorem 4]{LLMO21}. Note that without this assumption, Jensen's inequality gives the asymptotically slower rate
	\begin{align*} 
		\EE\big[\min_{l=1,...,k} \sum_{i=1}^n p_i f_i(x_l)\big] \leq \EE\big[\frac{1}{k} \sum_{l=1}^k \sum_{i=1}^n p_i f_i(x_l)\big]
		\leq
		\frac{c}{\sqrt{k}}
	\end{align*}
	for some constant $c>0$. We will need the following lemma.
	
	\begin{lem}{\cite[Lemma 3]{LVHL21}}
		\label{lem:LowerBound_stepsize_Lismooth}
		Let $\varphi$ be $\sigma$-strongly convex w.r.t. $\|\cdot\|$. Moreover, let the functions $f_i$ be $L$-smooth w.r.t. $\|\cdot\|$. Then it holds that 
		\[ t_{k,\sigma} \geq \frac{\sigma}{2L}.\]
	\end{lem}
	

	\begin{thm}
		\label{thm:Convergence_nonnegative_convex_Li_smooth}
		Let Assumptions~\ref{as:Standard_assumptions}-\ref{as:Group_of_assumptions_on_f_nonnegative_and_star_convex_part} hold true and assume that $\varphi$ is $\sigma$-strongly convex w.r.t. a norm $\|\cdot\|$ and all functions $f_i$ are $L$-smooth w.r.t. $\|\cdot\|$. Then the iterates $x_k$ of Algorithm~\ref{alg:NBK} fulfill that
		\begin{align*} 
			\EE\big[\min_{l=1,...,k} \sum_{i=1}^n p_i f_i(x_l)\big] \leq 
			\frac{4L}{\sigma k} \cdot \inf_{\hat x\in S} D^{x_0^*}_\varphi(x_0,\hat x).
		\end{align*}            
	\end{thm}
	
	\begin{proof}
		Combining Lemma~\ref{lem:Dphi_descent_estimate} and Lemma~\ref{lem:LowerBound_stepsize_Lismooth} yields that
		\[ D_\varphi^{x_{k+1}^*}(x_{k+1},\hat x) \leq D_\varphi^{x_k^*}(x_k,\hat x) - \frac{1}{2} f_{i_k}(x_k) \cdot t_{k,\sigma} \leq  D_\varphi^{x_k^*}(x_k,\hat x) - \frac{\sigma}{4L} f_{i_k}(x_k). \]
		The assertion now follows as in the proof of Theorem~\ref{thm:Convergence_nonnegative_convex_C1_leq} by taking expectation and telescoping.
	\end{proof} 
	
	For strongly star-convex functions $f_i$, we can prove a linear convergence rate, where we recover the contraction factor from~\cite[Theorem 3]{LLMO21}. Moreover, we can even improve this factor for smooth~$\varphi$.
	
	\begin{thm}
		\label{thm:Convergence_nonnegative_convex_Li_smooth_mu_sc}
		Let Assumptions~\ref{as:Standard_assumptions}-\ref{as:Group_of_assumptions_on_f_nonnegative_and_star_convex_part} hold true and assume that $\varphi$ is $\sigma$-strongly convex w.r.t. a norm $\|\cdot\|$ and all functions $f_i$ are $L$-smooth w.r.t. $\|\cdot\|$. Moreover, assume that $\overline f:=\sum_{i=1}^d p_if_i$ is $\mu$-strongly star-convex w.r.t. $S$ relative to~$\varphi$. Then there exists an element $\hat x \in S$ such that the iterates $x_k$ of Algorithm~\ref{alg:NBK} converge to $\hat x$ at the rate
		\[ \EE\big[D_\varphi^{x_{k+1}^*}(x_{k+1},\hat x)\big] \leq \big(1-\frac{\mu\sigma}{2L}\big) \EE\big[ D_\varphi^{x_k^*}(x_k,\hat x) \big] - \frac{\sigma}{4L}\overline f(x_k). \]
		Moreover, if $\varphi$ is $M$-smooth, it holds that
		\[ \EE\big[D_\varphi^{x_{k+1}^*}(x_{k+1},\hat x)\big] \leq  \big(1-\frac{\mu\sigma}{2L}-\frac{\mu\sigma^2}{4LM}\big)
		\EE\big[D_\varphi^{x_k^*}(x_k,\hat x)\big]. \]
	\end{thm}
	
	\begin{proof}
		By Lemma~\ref{lem:LowerBound_stepsize_Lismooth} we have that
		\[ t_k\big(\langle \nabla f_{i_k}(x_k),\hat x-x_k\rangle +  f_{i_k}(x_k)\big) 
		\leq \frac{\sigma}{2L}\big(\langle \nabla f_{i_k}(x_k),\hat x-x_k\rangle +  f_{i_k}(x_k)\big).
		\]
		Taking expectation and using the assumption of relative strong convexity, we obtain that
		\begin{align*} \EE\big[ t_k\big(\langle \nabla f_{i_k}(x_k),\hat x-x_k\rangle +  f_{i_k}(x_k)\big) \big] 
			&\leq - \frac{\sigma}{2L} \EE\big[{\overline f}(\hat x)-{\overline f}(x_k)-\langle\nabla{\overline f}(x_k),\hat x-x_k\rangle\big] \\
			&\leq -\frac{\mu\sigma}{2L}\EE\big[D_\varphi^{x_k^*}(x_k,\hat x)\big].
		\end{align*}
		The first convergence rate then follows by the steps in Lemma~\ref{lem:Dphi_descent_estimate} and Theorem~\ref{thm:Convergence_nonnegative_convex_Li_smooth}, replacing~\eqref{eqn:Estimate_leq_convexity_estimate} by the above inequality. Finally, let $\varphi$ be additionally $M$-smooth. Using that $\nabla \overline f(\hat x)=0$, we can further bound
		\begin{align*} \overline f(x_k) &= \overline f(x_k) - \overline f(\hat x) - \langle \nabla \overline f(\hat x), x_k-\hat x\rangle \\
			&\geq \mu D_\varphi(\hat x,x_k) \geq \frac{\mu\sigma}{2}\|x_k-\hat x\|^2\geq \frac{\mu\sigma}{M}D_\varphi^{x_k^*}(x_k,\hat x). 
		\end{align*}
	\end{proof}
	
	Since the proofs of Theorem~\ref{thm:Convergence_nonnegative_convex_C1_leq}, Theorem~\ref{thm:Convergence_nonnegative_convex_Li_smooth} and Theorem~\ref{thm:Convergence_nonnegative_convex_Li_smooth_mu_sc} rely on Lemma~\ref{lem:Dphi_descent_estimate}, they also hold for Algorithm \ref{alg:NBK_relaxed}.

	\subsection{Convergence under the local tangential cone condition}
	\label{subsec:tcc}
	
	Inspired by~\cite{WLBG22}, we consider functions fulfilling the so-called tangential cone condition, which was introduced in~\cite{HNS95} as a sufficient condition
	for convergence of the Landweber iteration for solving ill–posed nonlinear problems.
	
	\begin{defi}
		A differentiable function $f\colon D\to\RR$ fulfills the \emph{local tangential cone condition} \textup{($\eta$-TCC)} on $U\subset D$ with constant $0<\eta<1$, if for all $x,y\in U$ it holds that
		\begin{align} \label{eqn:tangential_cone_condition} |f(x)+\langle \nabla f(x),y-x\rangle -f(y)| \leq \eta |f(x)-f(y)|. \end{align}
	\end{defi}
	
	Under this condition, we are able to formulate a variant of Lemma~\ref{lem:Dphi_descent_estimate} and derive corresponding convergence rates. Precisely, we will assume the following.
	
	\begin{as}
		\label{as:tangential_cone_condition}
		There exist a point $\hat x\in S$ and constants $\eta\in ]0,1[$ and $r>0$ such that each function $f_i$ fulfills $\eta$-TCC w.r.t. $\eta$ on 
		\begin{align*}
			B_{r,\varphi}(\hat x) := \big\{x\in C: D_\varphi^{x^*}(x,\hat x) \leq r \quad \text{for all } x^*\in\partial\varphi(x) \big\}.
		\end{align*}
	\end{as}

	\begin{lem}
		\label{lem:Dphi_descent_estimate_tcc}
		Let Assumption~\ref{as:Standard_assumptions} and Assumption~\ref{as:tangential_cone_condition} hold true and assume that $\varphi$ is $\sigma$-strongly convex w.r.t. a norm~$\|\cdot\|$. Let $\hat x\in S$. Then, the iterates of Algorithm~\ref{alg:NBK} fulfill
		\[ D^{x_{k+1}^*}_\varphi(x_{k+1},\hat x) \leq D^{x_k^*}_\varphi(x_k,\hat x) - \tau \frac{\big(f_{i_k}(x_k)\big)^2}{\|\nabla f_{i_k}(x_k)\|_*^2}, \] 
		if one of the following conditions is fulfilled:
		\begin{enumerate}[(i)]
			\item $t_k=t_{k,\sigma}$, $\eta<\frac12$ and $\tau = \sigma\big(\frac{1}{2}-\eta\big)$,
			\item $t_k=t_{k,\varphi}$, $\varphi$ is additionally $M$-smooth w.r.t. $\|\cdot\|$, $\eta < \frac{\sigma}{2M}$ and $\tau = \sigma\big(\frac{1}{2}-\eta \frac{M}{\sigma}\big).$
		\end{enumerate}
		In particular, if $x_0\in B_{r,\varphi}(\hat x)$, then in both cases we have that $x_k\in B_{r,\varphi}(\hat x)$ for all $k\in\NN$.
	\end{lem}

	\begin{proof}
		
		For (i), by definition of $t_{k,\sigma}$ and $\eta$-TCC we have that
		\begin{align*}
			t_{k,\sigma} \big( f_{i_k}(x_k) + \langle f_{i_k}(x_k),\hat x-x_k\rangle \big) \leq \eta \sigma \frac{\big(f_{i_k}(x_k)\big)^2}{\|\nabla f_{i_k}(x_k)\|_*^2}.
		\end{align*}
		The first convergence rate then follows by the steps in Lemma~\ref{lem:Dphi_descent_estimate}, replacing~\eqref{eqn:Estimate_leq_convexity_estimate} by the above inequality. 
		For (ii), using Lemma~\ref{lem:Bounds_Bregman_SPS}(ii) and the fact that $f_{i_k}$ fulfills $\eta$-TCC, we estimate
		\begin{align*}
			t_{k,\varphi}\big(f_{i_k}(x_k)+\langle \nabla f_{i_k}(x_k), \hat x-x_k\rangle\big) 
			\leq \eta M\frac{f_{i_k}(x)^2}{\|\nabla f_{i_k}(x)\|_*^2},
		\end{align*}
		so that the assertion follows as in (i). 
	\end{proof}
	
	The condition on $\eta$ in (i) is the classical condition for Landweber methods (see e.g.~\cite[Theorem 3.8]{HNS95}) and is also required in the work~\cite{WLBG22}, which studies Algorithm~\ref{alg:NBK} for $\varphi(x)=\frac12\|x\|_2^2$ under $\eta$-TCC. The constant $\tau$ in (ii) can not be greater than $\tau$ in (i), as it holds that $\sigma\leq M$.

	\begin{thm} 
		\label{thm:Convergence_tcc_sublinear_rate}
		Let Assumption~\ref{as:Standard_assumptions} hold and let $\varphi$ be $\sigma$-strongly convex. Moreover, let Assumption~\ref{as:tangential_cone_condition} hold with some $\eta>0$ and $\hat x\in S$ and let $x_0\in B_{r,\varphi}(\hat x)$.
		\begin{enumerate}[(i)]
			\item If $\eta<\frac12$, then the iterates $x_k$ of Algorithm~\ref{alg:NBK_relaxed} converge a.s. to a random variable whose image is contained in the solution set $S\cap B_{r,\varphi}(\hat x)$ and it holds that
			\begin{align*} 
				\EE\big[\min_{l=1,...,k} \sum_{i=1}^n p_i \big(f_i(x_l)\big)^2\big] 
				\leq 
				\frac{C \cdot D_\varphi^{x_0^*}(x_0,\hat x)}{\sigma\big(\frac12-\eta\big) k}.
			\end{align*}
			\item Let $\varphi$ be additionally $M$-smooth and $\eta < \frac{\sigma}{2M}$. Assume that $x_k$ are the iterates of Algorithm~\ref{alg:NBK} and the condition $H_k\cap\dom\varphi\neq\emptyset$ is fulfilled for all $k$. Then the $x_k$ converge a.s. to a random variable whose image is contained in the solution set $S\cap B_{r,\varphi}(\hat x)$ and it holds that 
			\begin{align*} 
				\EE\big[\min_{l=1,...,k} \sum_{i=1}^n p_i \big(f_i(x_l)\big)^2\big] 
				\leq 
				\frac{C \cdot D_\varphi^{x_0^*}(x_0,\hat x)}{\sigma \big(\frac{1}{2}-\eta \frac{M}{\sigma}\big) k}. 
			\end{align*}
		\end{enumerate}
	\end{thm}
	
	\begin{proof}
		By Lemma~\ref{lem:Dphi_descent_estimate_tcc}, the $x_k$ stay in~$B_{r,\varphi}(\hat x)$.
		The statements now follow as in Theorem~\ref{thm:Convergence_nonnegative_convex_C1_leq} by invoking Lemma~\ref{lem:Dphi_descent_estimate_tcc} instead of Lemma~\ref{lem:Dphi_descent_estimate}.
	\end{proof}

	Finally, we can give a local linear convergence rate under the additional assumption that the Jacobian has full column rank. For $\varphi(x)=\frac{1}{2}\|x\|_2^2$, in part (i) of the theorem we recover the result from \cite[Theorem 3.1]{WLBG22} as a special case. In both Theorem~\ref{thm:Convergence_tcc_sublinear_rate} and Theorem~\ref{thm:Convergence_tcc_linear_rate}, unfortunately we obtain a more pessimistic rate for Algorithm~\ref{alg:NBK} compared to Algorithm~\ref{alg:NBK_relaxed}, as the $\tau$ in (ii) is upper bounded by the $\tau$ in (i).

	\begin{thm} 
	\label{thm:Convergence_tcc_linear_rate}
	Let Assumption~\ref{as:Standard_assumptions} hold true and let $\varphi$ be $\sigma$-strongly convex and $M$-smooth. Let Assumption~\ref{as:tangential_cone_condition} hold with some $\eta>0$ and $\hat x\in S$ and let $x_0\in B_{r,\varphi}(\hat x)$. Moreover, assume that the Jacobian $Df(x)$ has full column rank for all $x\in B_{r,\varphi}(\hat x)$ and $p_{\min}=\min_{i=1,...,n} p_i>0$. We set
	\begin{align*}
		\kappa_{\min}  := \min_{x\in B_{r,\varphi}(\hat x)} \min_{\|y\|_2=1} \frac{\|Df(x)\|_F}{ \|Df(x)y\|_2} .
		\end{align*}
		
		\begin{enumerate}
			\item[(i)] If $\eta<\frac12$, then the iterates $x_k$ of Algorithm~\ref{alg:NBK_relaxed} fulfill that 
			\begin{align} 
				\label{eqn:tcc_linear}
				\frac{\sigma}{2}\EE\big[\|x_k-\hat x\|_2^2] 
				\leq \EE\big[D_\varphi(x_{k},\hat x)\big] 
				\leq \Big( 1- \frac{ \sigma\big(\frac12-\eta\big)p_{\min}}{(1+\eta)^2\kappa_{\min}^2} \Big)^k \EE\big[D_\varphi(x_0,\hat x)\big].
			\end{align}
			
			\item[(ii)] Let $\eta < \frac{\sigma}{2M}$. Assume that $x_k$ are the iterates of Algorithm~\ref{alg:NBK} and the condition $H_k\cap\dom\varphi\neq\emptyset$ is fulfilled for all $k$. Then it holds that 
			\begin{align} 
				\label{eqn:tcc_linear_Alg2}
				\frac{\sigma}{2}\EE\big[\|x_k-\hat x\|_2^2] 
				\leq \EE\big[D_\varphi(x_{k},\hat x)\big] 
				\leq \Big( 1- \frac{ \sigma\big(\frac12-\eta\frac{M}{\sigma}\big)p_{\min}}{M(1+\eta)^2\kappa_{\min}^2} \Big)^k \EE\big[D_\varphi(x_0,\hat x)\big].
			\end{align}
		\end{enumerate}
	\end{thm}
	
	For the proof, as in~\cite{WLBG22} we use the following auxiliary lemma.
	
	\begin{lem}
		\label{lem:quot_diff_estimate}
		Let $a_1,...,a_n\geq 0$ and $b_1,...,b_n>0$. Then it holds that 
		\[ \sum_{i=1}^n \frac{a_i}{b_i} \geq \frac{\sum_{i=1}^n a_i}{\sum_{i=1}^n b_i}. \] 
	\end{lem}
	\begin{proof}
		Since $a_{i},b_{i}\geq 0$ it holds that 
		\begin{align*}
			\Big( \sum_{i=1}^n b_i \Big) \Big(  \sum_{j=1}^n \frac{a_j}{b_j} \Big) 
			= \sum_{i,j=1}^n b_i \frac{a_j}{b_j}
			\geq \sum_{i=1}^n b_i \frac{a_i}{b_i} = \sum_{i=1}^n a_i.
		\end{align*}
	\end{proof}

	\begin{proof}[Proof of Theorem~\ref{thm:Convergence_tcc_linear_rate}] 
		By Assumption \ref{as:tangential_cone_condition} and the fact that $f(\hat x)=0$, we can estimate 
		\begin{align*}
			|\langle \nabla f_{i_k}(x_k), \hat x - x_k\rangle| 		
			&\leq |f_{i_k}(x_k) + \langle \nabla f_{i_k}(x_k), \hat x - x_k\rangle - f_{i_k}(\hat x)| + |f_{i_k}(x_k)-f_{i_k}(\hat x)| \\
			&\leq (1+\eta) |f_{i_k}(x_k)-f_{i_k}(\hat x)| = (1+\eta) |f_{i_k}(x_k)|.
		\end{align*}
		In all cases of the assumption, inserting the above estimate we respectively conclude that 
		\begin{align*} 
			D_\varphi(x_{k+1},\hat x) 
			\leq D_\varphi(x_k,\hat x) - \frac{\tau}{(1+\eta)^2} \cdot \frac{|\langle \nabla f_{i_k}(x_k), \hat x-x_k\rangle|^2}{\|\nabla f_i(x_k)\|_2^2}.
		\end{align*}
		
		Taking expectation and using Lemma~\ref{lem:quot_diff_estimate} as well as the definition of $\kappa_{\min}$, we conclude that
		\begin{align*} 
			\EE\big[D_\varphi(x_{k+1},\hat x)\big]  &\leq 
			\EE\big[D_\varphi(x_{k},\hat x)\big] - \frac{\tau}{(1+\eta)^2} \cdot \EE\Big[ \sum_{i=1}^n p_i \frac{|\langle \nabla f_i(x_k), \hat x-x_k\rangle|^2}{\|\nabla f_i(x_k)\|_2^2} \Big] \\
			&\leq 
			\EE\big[D_\varphi(x_{k},\hat x)\big] - \frac{\tau p_{\min}}{(1+\eta)^2}  \cdot \EE\Big[\frac{\|Df(x_k)(\hat x-x_k)\|_2^2}{\|Df(x_k)\|_F^2} \Big]\\
			&\leq \EE\big[D_\varphi(x_{k},\hat x)\big] - \EE\Big[ \frac{\tau p_{\min}}{(1+\eta)^2 \kappa_{\min}^2} \cdot \|x_k-\hat x\|_2^2 \Big] \\
			&\leq \Big( 1- \frac{ 2\tau p_{\min}}{(1+\eta)^2M \kappa_{\min}^2} \Big) \EE\big[D_\varphi(x_{k},\hat x)\big].
		\end{align*}
	\end{proof}

\section{Numerical experiments}

In this section, we evaluate the performance of the NBK method. In the first experiment we used NBK to find sparse solutions with the nonsmooth DGF $\varphi(x) = \frac{1}{2}\|x\|_2^2+\lambda\|x\|_1$ for unconstrained quadratic equations, that is, with $C=\RR^d$. Next, we employed the negative entropy DGF over the probability simplex $C=\Delta^{d-1}$ to solve simplex-constrained linear equations as well as the \emph{left-stochastic decomposition problem}, a quadratic problem over a product of probability simplices with applications in clustering~\cite{AFGK13}. All the methods were implemented in MATLAB on a macbook with 1,2 GHz Quad-Core Intel Core i7 processor and 16 GB memory. Code is available at \url{https://github.com/MaxiWk/Bregman-Kaczmarz}.

\subsection{Sparse solutions of quadratic equations}
As the first example, we considered multinomial quadratic equations
\[ f_i(x) = \frac12\langle x,A^{(i)}x\rangle + \langle b^{(i)},x\rangle + c^{(i)}=0 \]
with $A^{(i)}\in\RR^{d\times d}$, $b^{(i)}\in\RR^d$, $c^{(i)}\in\RR$ and $i=1,...,n$. We investigated if Algorithm~\ref{alg:NBK} (NBK method) and Algorithm~\ref{alg:NBK_relaxed} (rNBK method) are capable of finding a sparse solution $\hat x\in\RR^d$ by using the DGF $\varphi(x)=\lambda\|x\|_1+\frac12\|x\|_2^2$ and tested both methods against the euclidean nonlinear Kaczmarz method~(NK). As it holds $\dom \varphi = \RR^d$, it is always possible to choose the step size $t_{k,\varphi}$ from~\eqref{eqn:BregProj_stepsize} in the NBK method. Moreover, the step size can be computed exactly by a sorting procedure, as $\varphi^*$ is a continuous piecewise quadratic function, see Example~\ref{ex:varphi_from_sparse_Kaczmarz}. In order to guarantee existence of a sparse solution, we chose a sparse vector $\hat x\in\RR^d$, sampled the data $A^{(i)}, b^{(i)}$ randomly with entries from the standard normal distribution and set 
\[ c^{(i)} = - \Big( \frac12\langle x,A^{(i)}x\rangle + \langle b^{(i)},x\rangle \Big). \]
In all examples, the nonzero part of $\hat x$ and the initial subgradient $x_0^*$ were sampled from the standard normal distribution. The initial vector $x_0$ was computed by $x_0=\nabla\varphi^*(x_0^*) = S_\lambda(x_0^*)$.

From the updates it is evident that computational cost per iteration is cheapest for the NK method, slightly more expensive for the rNBK method and most expensive for the NBK method. To examine the case $d<n$, we chose $A^{(i)}\sim\mathcal N(0,1)^{500\times 500}$ for $i=1,...,1000$, $\hat x$ with $25$ nonzero entries and $\lambda=10$ and performed 20 random repeats. Figure~\ref{fig:sparse_quadratic_overdet} shows that the NBK method clearly outperforms the other two methods in this situation, even despite the higher cost per iteration. 

Figure~\ref{fig:sparse_quadratic_underdet} illustrates that in the case $d>n$, both the NBK method and the rNBK method can fail to converge or converge very slowly.

\begin{figure}[htb]
	\begin{center}

		\begin{tabular}{rl} 
			\includegraphics{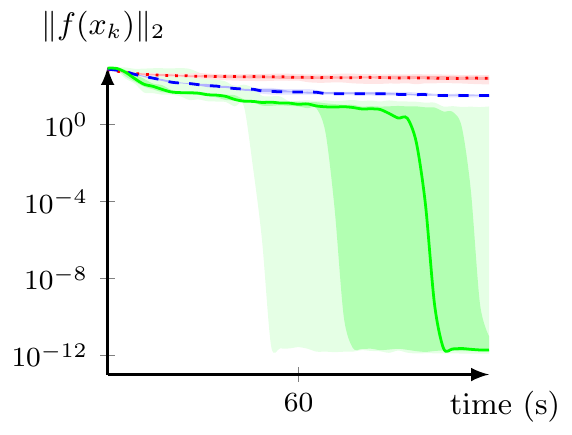}
			&
			\includegraphics{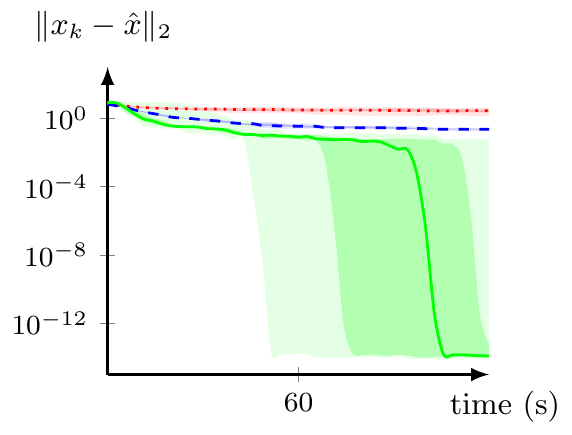}
		\end{tabular}
		\includegraphics{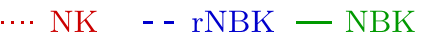}
		
	\end{center}
	\caption{Experiment with quadratic equations, $(n,d) = (1000,500)$, $\hat x$ with $50$ nonzero entries, 20 random repeats. Left: plot of residual $\|f(x_k)\|_2$, right: plot of distance to solution $\hat x$, both over computation time. Thick line shows median over all trials, light area is between min and max, darker area indicates 25th and 75th quantile.}
	\label{fig:sparse_quadratic_overdet}
\end{figure}

\begin{figure}[htb]
	\begin{center}

		\begin{tabular}{rl} 			
			\includegraphics{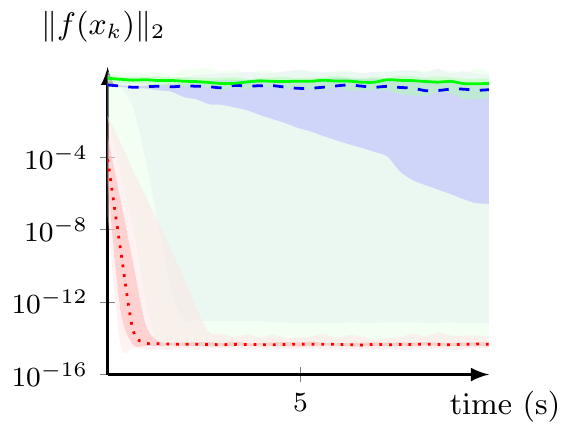}
			&
			\includegraphics{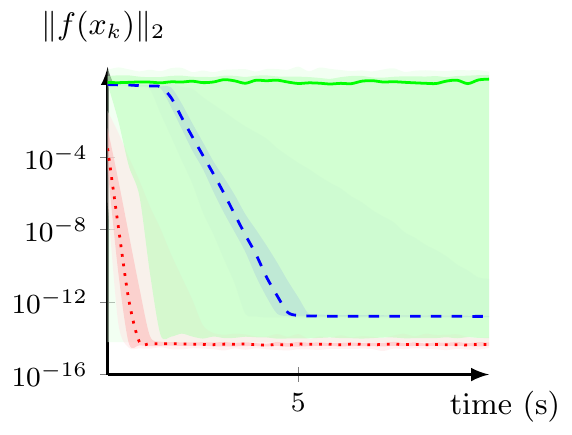}
		\end{tabular}
		
		\includegraphics{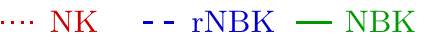}
		
	\end{center}
	\caption{Experiment with quadratic equations, $(n,d) = (50,100)$, $\hat x$ with $5$ nonzero entries, $50$ random repeats, plot of residual $\|f(x_k)\|_2$ against computation time. Left: $\lambda=2$, right: $\lambda=5$. Thick line shows median over all trials, light area is between min and max, darker area indicates 25th and 75th quantile.}
	\label{fig:sparse_quadratic_underdet}
\end{figure}

\subsection{Linear systems on the probability simplex}
We tested our method on linear systems constrained to the probability simplex
\begin{align}
	\label{eqn:linear_system_on_probability_simplex} 
	\text{find } x\in\Delta^{d-1}: \qquad Ax=b.
\end{align}
That is, in problem~\eqref{eqn:problem} we chose $f_i=\langle a_i,x\rangle - b_i$ with $D=\RR^d$ and viewed $C=\Delta^{d-1}$ as the additional constraint. For Algorithm~\ref{alg:NBK}, we used the simplex-restricted negative entropy function from Example~\ref{ex:Simplex_entropy}, i.e. we set \[ \varphi(x) = \begin{cases}
	\sum_{i=1}^d x_i \log(x_i), & x\in\Delta^{d-1} , \\
	+\infty, & \text{otherwise.}
\end{cases} \]
We know from Example~\ref{ex:Simplex_entropy} that $\varphi$ is $1$-strongly convex w.r.t. the $1$-norm $\|\cdot\|_1$. Therefore, as the second method we considered the rNBK iteration given by Algorithm~\ref{alg:NBK_relaxed} with $\sigma=1$ and $\|\cdot\|_*=\|\cdot\|_\infty$. As a benchmark, we considered a POCS (orthogonal projection) method which computes an orthogonal projection onto a row equation, followed by an orthogonal projection onto the probability simplex, see Algorithm~\ref{alg:EuclProj_linear_system_probability_simplex} listed below. We note that in \cite[Theorem 3.3]{KS21} it has been proved that the distance of the iterates of the POCS method and the NBK method to the set of solutions on $\Delta^{d-1}_+$ decays with an expected linear rate, if there exists a solution in $\Delta^{d-1}_+$. Theorem~\ref{thm:Convergence_nonnegative_convex_C1_leq} shows at least a.s. convergence of the iterates towards a solution for all three methods.

We note that it holds $\nabla f_{i_k}(x)=a_{i_k}$ for all $x$ and $\beta_k = b_{i_k}$ in the NBK method. If problem~\eqref{eqn:linear_system_on_probability_simplex} has a solution, then condition~\eqref{eqn:Condition_hyperplane_nonempty_intersection} is fulfilled in each step of the NBK method, so the method takes always the step size $t_k=t_{k,\varphi}$ from the exact Bregman projection. For the projection onto the simplex in Algorithm~\ref{alg:EuclProj_linear_system_probability_simplex}, we used the pseudocode from~\cite{PW13}, see also~\cite{Brucker84,DS08,BG84}. 

\begin{algorithm}[htb]
	\begin{algorithmic}[1]
		\State Input: probabilities $p_i>0$ for $i=1,...,n$ 
		\State Initialization: $x_0\in\Delta^{d-1}$
		\For{$k=0,1,...$}
		\State choose $i_k\in\{1,...,n\}$ according to the probabilities $p_1,...,p_n$
		\State project $y_{k+1} = \Pi_{H(a_i,b_i)}(x_k) = x_k - \frac{\langle a_{i_k},x_k\rangle-b_{i_k}}{\|a_{i_k}\|_2^2} a_{i_k}$
		\State project $x_{k+1} = \Pi_{\Delta^{d-1}}(y_{k+1})$
		\EndFor
	\end{algorithmic}
	\caption{Alternating euclidean projections (POCS method) for~\eqref{eqn:linear_system_on_probability_simplex}}
	\label{alg:EuclProj_linear_system_probability_simplex}
\end{algorithm}

In our experiments we noticed that in large dimensions, such as $d\geq 100$, solving the Bregman projection~\eqref{eqn:BregProj_stepsize} up to a tolerance of $\epsilon=10^{-9}$ takes less than half as much computation time as the simplex projection. As these two are the dominant operations in these methods, the NBK updates are computationally cheaper than the NK updates in the high dimensional setting. However, the examples will show that convergence quality of the methods depends on the distribution of the entries of $A$. All methods were observed to converge linearly.

In the following experiments, we took different choices of $A$ and set the right-hand side to $b=A\hat x$ with a point $\hat x$ drawn from the uniform distribution on the probability simplex $\Delta^{d-1}$. All methods were initialized with the center point $x_0=(\frac{1}{d}, ..., \frac{1}{d}).$ 

For our first experiment, we chose standard normal entries $A\sim\mathcal N(0,1)^{n\times d}$ with $(n,d) = (500,200)$ and $(n,d)=(200,500)$. Figure~\ref{fig:simplex_linear_stdn_distr} shows that in this setting, the POCS method achieves much faster convergence in the overdetermined case $(n,d) = (500,200)$ than the NBK method, whereas both methods perform roughly the same in the underdetermined case $(d,n) = (200,500)$. The rNBK method is considerably slower than the other two methods, which shows that the computation of the $t_{k,\varphi}$ step size for NBK pays off. 

In our second experiment, we built up the matrix from uniformly distributed entries $A\sim\mathcal U([0,1])^{n\times d}$ and $A\sim\mathcal U([0.9,1])^{n\times d}$ with $(n,d)=(200,500)$. The results are summarized in Figure~\ref{fig:simplex_linear_unif_distr}. For the Kaczmarz method it has been observed in practice that so called 'redundant' rows of the matrix $A$ deteriorate the convergence of the method~\cite{JNY22}. This effect can also occur with the POCS method, as it also relies on euclidean projections. Remarkably, we can see that this is not the case for the NBK method and it clearly outperforms the POCS method and the rNBK method. This in particular shows that the multiplicative update used in both the rNBK method and the NBK method is not enough to overcome the difficulty of redundancy- to achieve fast convergence, it must be combined with the appropriate step size which is used by the proposed NBK method.

Finally, we illustrate the effect of the accuracy $\epsilon$ in step size computation for the NBK method. We chose $A\sim U([0,1])^{n\times d}$ and $\epsilon=10^{-9}$ and compared with the larger tolerance $\epsilon=10^{-5}$. Figure~\ref{fig:simplex_linear_effect_of_inaccuracy} shows that, with $\epsilon=10^{-5}$, the residual plateaus 
at a certain threshold. In contrast with $\epsilon=10^{-9}$, the residual does not plateau, and despite the more costly computation of the step size, the NBK method is still competitive with respect to time.
Hence, for the problem of linear equations over the probability simplex we recommend to solve the step size problem up to high precision.


\begin{figure}[htp]
	\begin{center}

		\begin{tabular}{m{0.5\textwidth}m{0.5\textwidth}}
					\hspace{-.75cm}
			\includegraphics{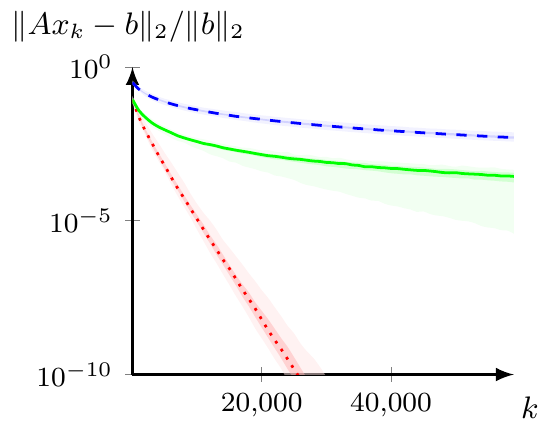}
			&
					\hspace{-.75cm}
			\includegraphics{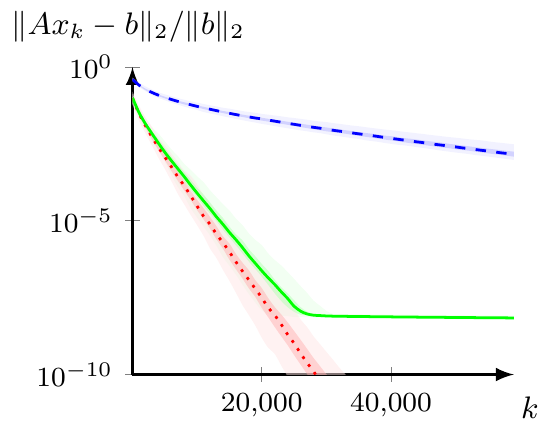}
			\\
					\hspace{-.75cm}
			\includegraphics{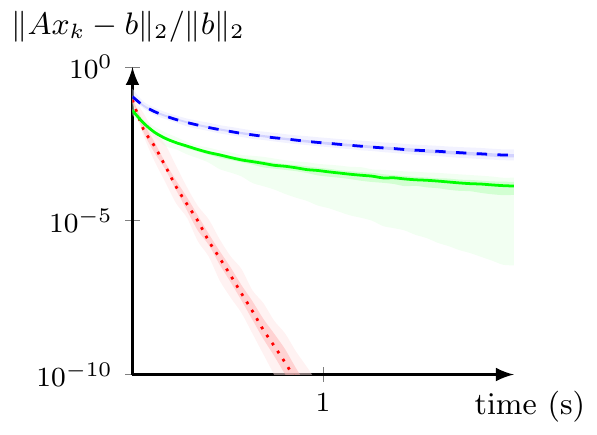}
			&
					\hspace{-.75cm}
			\includegraphics{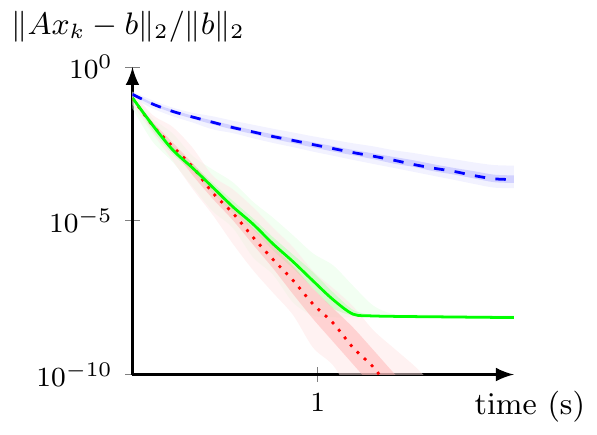}
		\end{tabular}
		
		\includegraphics{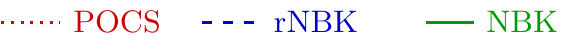}
		
		
		
	\end{center}
	\caption{Experiment with linear equations on the probability simplex, plot of relative residuals averaged over 50 random examples against iterations ($k$) and computation time. Left column: $A\sim\mathcal N(0,1)^{500\times 200}$, right column: $A\sim\mathcal N(0,1)^{200\times 500}$. Thick line shows median over all trials, light area is between min and max, darker area indicates 25th and 75th quantile.}
	\label{fig:simplex_linear_stdn_distr}
\end{figure}


\begin{figure}[htb]
	\begin{center}

		\begin{tabular}{m{0.5\textwidth}m{0.5\textwidth}}
					\hspace{-.75cm}
			\includegraphics{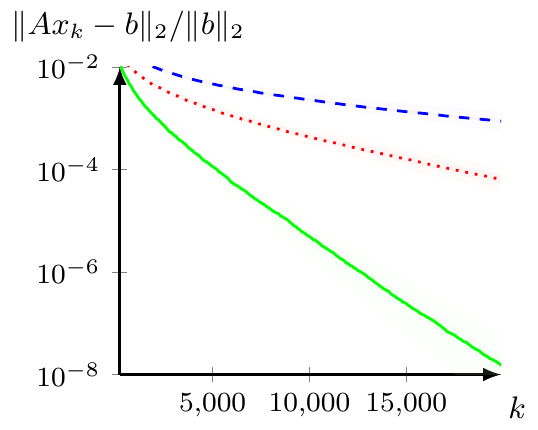}
			&
					\hspace{-.75cm}
			\includegraphics{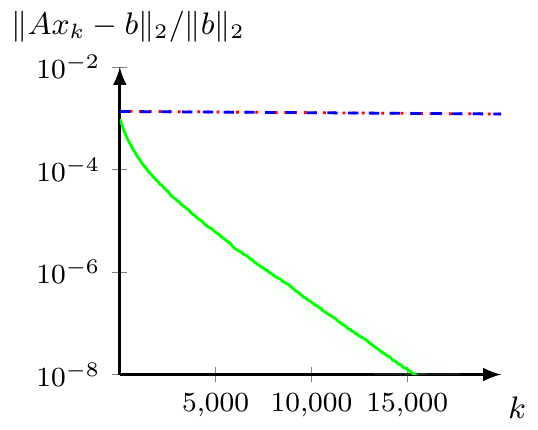}
			\\
					\hspace{-.75cm}
			\includegraphics{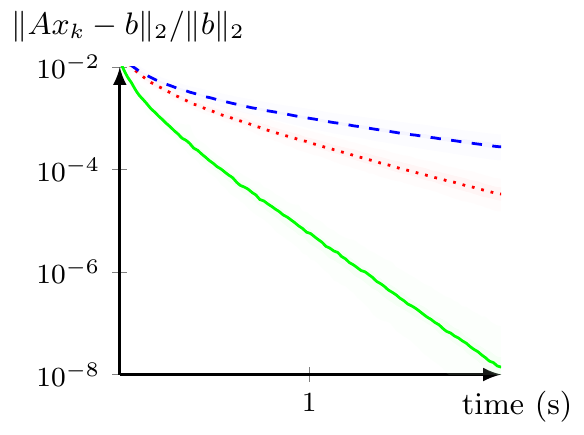}
			&
					\hspace{-.75cm}
			\includegraphics{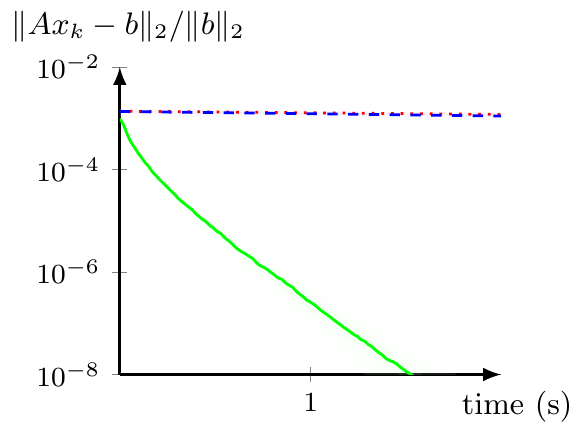}
		\end{tabular}
		
		\includegraphics{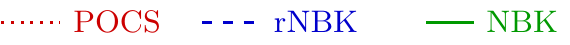}
		
	\end{center}
	\caption{Experiment with linear equations on the probability simplex, plot of relative residuals averaged over 50 random examples against iterations ($k$) and computation time. Left column: $A\sim\mathcal U([0,1])^{200\times 500}$, right column: $A\sim\mathcal U([0.9,1])^{200\times 500}$. Thick line shows median over all trials, light area is between min and max, darker area indicates 25th and 75th quantile.}
	\label{fig:simplex_linear_unif_distr}
\end{figure}


\begin{figure}[htb]
	\begin{center}

		\hspace{-.75cm}
		\begin{tabular}{rl} 			
			\includegraphics{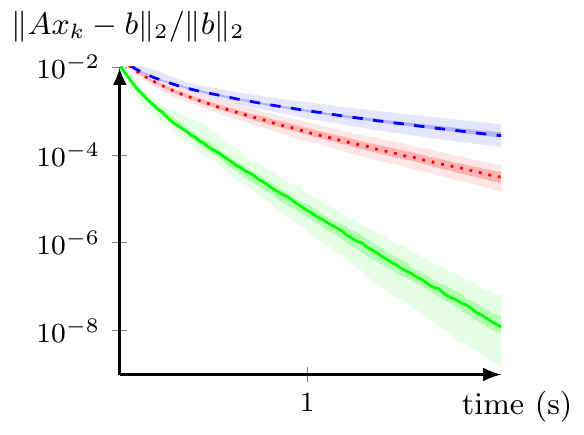}
			&
			\includegraphics{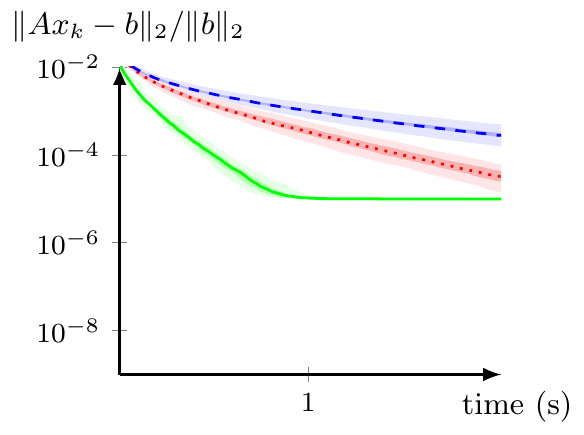}
		\end{tabular}
		
		\includegraphics{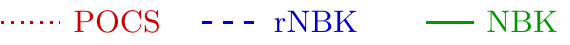}
		
	\end{center}
	\caption{Experiment with linear equations on the probability simplex, plot of relative residuals averaged over 50 random examples against computation time. In both examples, $A\sim\mathcal U([0,1])^{200\times 500}$. Left: $\epsilon=10^{-9}$, right: $\epsilon=10^{-5}$ in NBK method. Thick line shows median over all trials, light area is between min and max, darker area indicates 25th and 75th quantile.}
	\label{fig:simplex_linear_effect_of_inaccuracy}
\end{figure}

\subsection{Left stochastic decomposition}
\label{subsec:LSD}

The \emph{left stochastic decomposition} (LSD) problem can be formulated as follows:
\begin{align}
	\label{eqn:LSD_problem}
	\text{find } X\in L^{r\times m}: \qquad X^TX = A,
\end{align}
where \[L^{r,m}:=\big\{P\in\RR^{r\times m}_{\geq 0}: P^T\mathbbm{1}_r = \mathbbm{1}_m\}\] is the set of left stochastic matrices and $A\in\RR^{r\times m}$ is a given nonnegative matrix. The problem is equivalent to the so-called \emph{soft-K-means} problem and hence has applications in clustering~\cite{AFGK13}. We can view~\eqref{eqn:LSD_problem} as an instance of problem~\eqref{eqn:problem} with component equations
\begin{align*}
	f_{i,j}(X) = \langle X_{:,i}, X_{:,j} \rangle - A_{i,j} = 0 \qquad \text{for } i=1,...,r, \ j=1,...,m
\end{align*}
and $C=L^{r\times m}\cong \left(\Delta^{r-1}\right)^m$, where $X_{:,i}$ denotes the $i$th column of $X$. For Algorithm~\ref{alg:NBK} we chose the DGF from Example~\ref{ex:DGF_and_NBK_for_cartesian_products} with the simplex-restricted negative entropy $\varphi_i=\varphi$ from Example~\ref{ex:Simplex_entropy}.
Since $f_{i,j}$ depends on at most two columns of $X$, Algorithm~\ref{alg:NBK} acts on $\Delta^{r-1}$ or $\Delta^{r-1}\times \Delta^{r-1}$ in each step. Therefore, we applied the steps from Example~\ref{ex:Simplex_entropy} in the first case, and from Example~\ref{ex:Double_simplex_entropy} in the second case.

We compared the performance of Algorithm~\ref{alg:NBK} and Algorithm~\ref{alg:NBK_relaxed} to a projected nonlinear Kaczmarz method given by Algorithm~\ref{alg:EuclProj_LSD}. Here, by $X_{k,:,i}$ we refer to the $i$th column of the $k$th iterate matrix. In all examples, we set $A=\hat X^T \hat X$, where the columns of $\hat X$ were sampled according to the uniform distribution on~$\Delta^{r-1}$.

\begin{algorithm}[htb]
	\begin{algorithmic}[1]
		\State Input: $\sigma>0$ and probabilities $p_{ij}$ for $i=1,...,r$ and $j=1,...,m$
		\State Initialization: $X_0\in L^{r\times m}$
		\For{$k=0,1,...$}
		\State choose $i_k\in\{1,...,r\}$ and $ j_k\in \{1,...,m\}$ according to $p_{1r},...,p_{rm}$
		\State set $\beta_k = \langle \nabla f_{i_k}(x_k),x_k\rangle - f_{i_k}(x_k) = \langle X_{k,:,i_k}, X_{k,:,j_k}\rangle + A_{i_k,j_k}$
		\If{$i_k=j_k$}
		\State project $Y_{k+1,:,i_k} = \Pi_{H(\alpha_k,\beta_k)}       
		X_{k,:,i_k}$ with $\alpha_k = 2X_{k,:,i_k}$ 
		\State project $X_{k+1,:,i_k} = \Pi_{\Delta^{m-1}}(Y_{k+1,:,i_k})$
		\EndIf
		\If{$i_k\neq j_k$}
		\State set $t_k = \frac{\langle X_{k,:,i_k}, X_{k,:,j_k} \rangle - A_{i_k,j_k}}{\|X_{k,:,i_k}\|_2^2 + \|X_{k,:,j_k}\|_2^2}$ \vspace{0.1cm}
		\State set $Y_{k+1,:,i_k} = Y_{k,:,i_k} - t_k Y_{k,:,j_k}$ 
		\State set $Y_{k+1,:,j_k} = Y_{k,:,j_k} - t_k Y_{k,:,i_k}$ 
		\State project $X_{k+1,:,i_k} = \Pi_{\Delta^{r-1}}(Y_{k+1,:,i_k})$
		\State project $X_{k+1,:,j_k} = \Pi_{\Delta^{r-1}}(Y_{k+1,:,j_k})$
		\EndIf 
		\EndFor
	\end{algorithmic}
	\caption{Projected nonlinear Kaczmarz method (PNK) for~\eqref{eqn:LSD_problem}}
	\label{alg:EuclProj_LSD}
\end{algorithm}

We observed that Algorithm~\ref{alg:NBK} (NBK method) gives the fastest convergence, if $r$ is not much smaller than $m$, see Figure~\ref{fig:LSD_100x50} and Figure~\ref{fig:LSD_50x100}. In both experiments, we noticed that condition~\eqref{eqn:Condition_hyperplane_nonempty_intersection} was actually fulfilled in each step, but checking did not show a notable difference in performance. The most interesting setting for clustering is that $r$ is very small and $m$ is large, as $r$ is the number of clusters~\cite{AFGK13}. However, it appears unclear if the NBK or the PNK method is a better choice for this problem size, as Figure~\ref{fig:LSD_3x100} shows. In this experiment, condition \eqref{eqn:Condition_hyperplane_nonempty_intersection} was not always fulfilled in the NBK method and we needed to employ the globalized Newton method together with an additional condition to the step size approximation, see Appendix for details. Finally, we can again see that Algorithm~\ref{alg:NBK_relaxed} is clearly outperformed by the other two methods in all experiments.

\begin{figure}[htb]
	
	\begin{center}

		\begin{tabular}{rl} 			
			\includegraphics{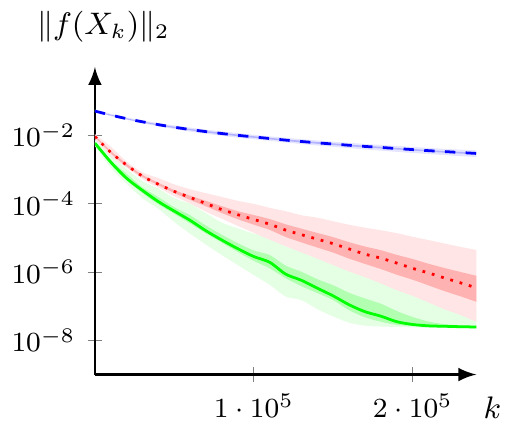}
			&
			\includegraphics{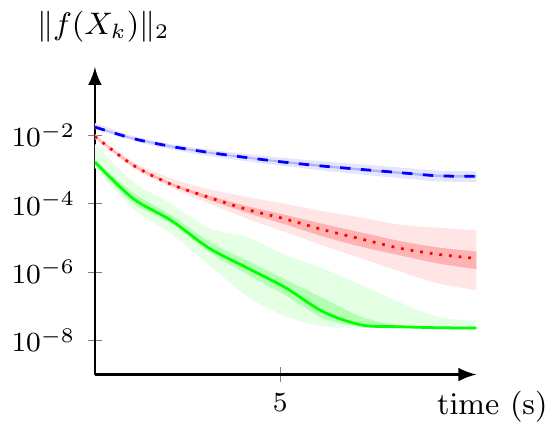}
		\end{tabular}
		
		\includegraphics{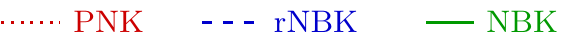}
		
		\vspace{0.3cm}
		
		\caption{Experiment 'Left stochastic decomposition problem' with $r=100, m=50$. Residuals $\|f(X^{(k)})\|_2$ averaged over 50 random examples against outer iterations $k$ (left) and computation time (right). Thick line shows median over all trials, light area is between min and max, darker area indicates 25th and 75th quantile.}
		\label{fig:LSD_100x50}
		
	\end{center}
	
\end{figure}

\begin{figure}[htb]
	
	\begin{center}

		\begin{tabular}{rl} 			
			\includegraphics{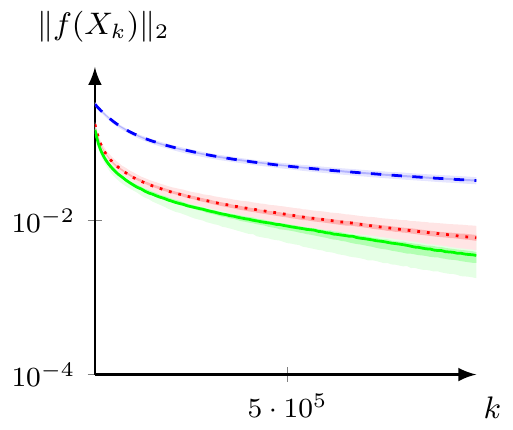}
			&
			\includegraphics{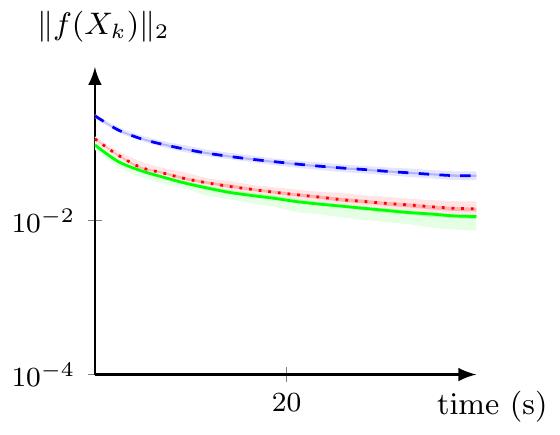}
		\end{tabular}
		
		\includegraphics{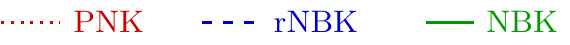}
		
		\vspace{0.3cm}
		
		\caption{Experiment 'Left stochastic decomposition problem' with $r=50, m=100$, plot of residuals $\|f(X^{(k)})\|_2$ averaged over 50 random examples against iterations (left) and computation time (right). Thick line shows median over all trials, light area is between min and max, darker area indicates 25th and 75th quantile.}
		\label{fig:LSD_50x100}
		
	\end{center}
	
\end{figure}

\begin{figure}[htb]
	
	\begin{center}

		\begin{tabular}{rl} 			
			\includegraphics{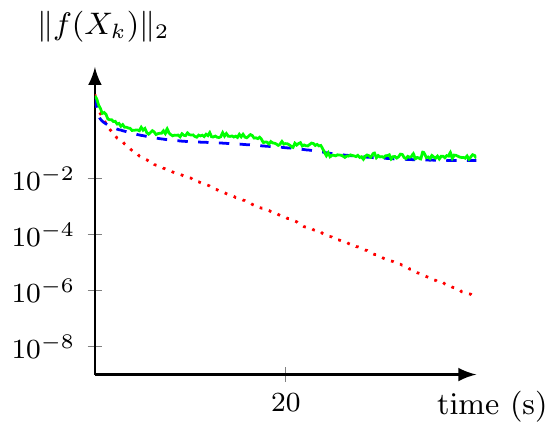}
			&
			\includegraphics{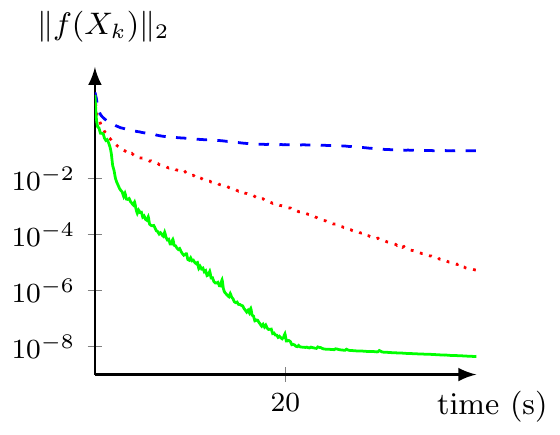} 
		\end{tabular}
		
		\includegraphics{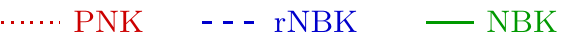}
		
		\vspace{0.3cm}
		
		\caption{Experiment 'Left stochastic decomposition problem' with $r=3, m=100$, residuals $\|f(X_k)\|_2$ against computation time. Left and right: Two random examples with different convergence behavior. Thick line shows median over all trials, light area is between min and max, darker area indicates 25th and 75th quantile.}
		\label{fig:LSD_3x100}
		
	\end{center}
	
\end{figure}

\section{Conclusions and further research}

We provided a general Bregman projection method for solving nonlinear equations, where each iteration needs only to sample one equation to make progress towards the solution. As such, the cost of one iteration scales independently of the number of equations. Our method is also a generalization of the nonlinear Kaczmarz method which allows for additional simple constraints or sparsity inducing regularizers. We provide two global convergence theorems under different settings and find a number of relevant experimental settings where instantiations of our method are efficient. 

Convergence for non-strongly convex distance generating functions $\varphi$, as well as a suitable scope of $\sigma$ in this setting, has so far not been explored.

Our work also opens up the possibility of incorporating more structure into SGD type methods in the interpolation setting as has been done in~\cite{JLZ23} for the linear case. In this setting each $f_i(x)$ is a positive loss function over the $i$th data point. If we knew in addition that some of the coordinates of $x$ are meant to be positive, or that $x$ is a discrete probability measure, then our nonlinear Bregman projection methods applied to the interpolation equations would provide new adaptive step sizes for stochastic mirror descent. Further venues for exploring would be to relax the interpolation equations, say into inequalities~\cite{slackpolyak2022}, and applying an analogous Bregman projections to incorporate more structure. We will leave this to future work.


\appendix

\section{Newton's method for line search problem \eqref{eqn:BregProj_stepsize}}

We compute the Newton update for problem~\eqref{eqn:BregProj_stepsize} for general $\varphi$ with $C^2$-smooth conjugate $\varphi^*$. 
The function $g_{i_k,x_k^*}$ from \eqref{eqn:g_function}
has first derivative
\begin{align*}
	g_{i_k,x_k^*}'(t) &= \big\langle\nabla\varphi^*(x_k^*-t\nabla f_{i_k}(x_k)), -\nabla f_{i_k}(x_k)\big\rangle + \beta_k \\
	&=\big\langle x_k - \nabla\varphi^*(x_k^* - t\nabla f_{i_k}(x_k)), \ \nabla f_{i_k}(x_k) \big\rangle - f_{i_k}(x_k)
\end{align*}
and second derivative
\begin{align*} 
	g_{i_k,x_k^*}''(t) = \big\langle \nabla^2 \varphi^*(x_k^* - t\nabla f_{i_k}(x_k)) \nabla f_{i_k}(x_k), \ \nabla f_{i_k}(x_k) \big\rangle \geq 0.
\end{align*}
If it holds $g_{f_{i_k},x_k^*}''(t)>0$, Newton's method for~\eqref{eqn:BregProj_stepsize} reads  
\begin{align*}
	t_{k,l+1} = t_{k,l} - \frac{g_{i_k,x_k^*}'(t_{k,l})}{g_{i_k,x_k^*}''(t_{k,l})}.
\end{align*}
As an initial value we use the step size $t_{k,0}:= \frac{ f_{i_k}(x_k)}{\|\nabla f_{i_k}(x_k)\|_2^2}$ from the $\ell_2$-projection of $x_k$ onto $H_k$. We propose to stop the method if $|g_{i_k,x_k^*}'(t_{k,l})|<\epsilon$. Typical values we used for our numerical examples were $\epsilon\in\{ 10^{-5}, 10^{-6}, 10^{-9}, 10^{-15}\}$. \\
It may happen that problem~\eqref{eqn:BregProj_stepsize} is ill-conditioned, in which case the Newton iterates $t_{k,l}$ may diverge quickly to $\pm\infty$ or alternate between two values. We have observed this can e.g. happen for the problem on left stochastic decomposition in Subsection~\ref{subsec:LSD}, if the number $m$ of rows of the matrix $X$ in the problem is small. 

In case that the Newton method diverges, we used the recently proposed globalized Newton method from~\cite{Mishchenko21}, which reads
\[ t_{k,l+1} = t_{k,l} - \frac{g_{i_k,x_k^*}'(t_{k,l})}{H\cdot \sqrt{|g_{i_k,x_k^*}'(t_{k,l})|} + g_{i_k,x_k^*}''(t_{k,l})} \]
with a fixed constant $H>0$. Also here, we stop if $|g_{i_k,x_k^*}'(t_{k,l})|<\epsilon$.
Convergence of the $t_{k,l}$ for $l\to\infty$ is guaranteed, if $\varphi^*$ is strongly convex, i.e. if $\varphi$ is everywhere finite with Lipschitz continuous gradient and the values $g_{i_k,x_k^*}(t_{k,l})$ are guaranteed to converge to the minimum value if $\varphi^*$ has Lipschitz continuous Hessian~\cite{Mishchenko21}. We have also observed good convergence for the negative entropy function on $\RR_{\geq 0}^d$ with this method when Newton's method is unstable. For problems constrained to the probability simplex $\Delta^{d-1}$, the globalized Newton method converged more slowly than the vanilla Newton method. For the problem in subsection~\ref{subsec:LSD} with $(r,m)=(3,100)$ we chose $H=0.1$. In addition, we performed a relaxed Bregman projection (line 10 of Algorithm \ref{alg:NBK}) with step size~\eqref{eqn:mSPS_like_stepsize} if $|t_{k,l}|>100$.

	\section*{Declarations}
\subsection*{Competing interests}
The authors have no competing interests to declare that are relevant to the content of this article.
\subsection*{Data availability}
We do not analyze or generate any datasets, because our work proceeds
within a theoretical and mathematical approach. However, the code that
generates the figures in this article can be found at \url{https://github.com/MaxiWk/Bregman-Kaczmarz}.

\bibliography{refs}

\bibliographystyle{abbrv}

\end{document}